\documentclass[11pt,oneside,reqno]{amsart}

\usepackage{amsmath,amssymb,amsthm,textcomp}
\usepackage{amsfonts,graphicx}
\usepackage[mathscr]{eucal}
\usepackage{color}
\usepackage{csquotes}
\usepackage{enumitem}
\usepackage{mathtools}
\numberwithin{equation}{section}

\theoremstyle{definition}

\numberwithin{equation}{section}

\newtheorem{theorem}{\bf Theorem}[section]
\newtheorem{remark}{\bf Remark}[section]
\newtheorem{proposition}{Proposition}[section]

\newtheoremstyle
{remarkstyle}
{}
{11pt}
{}
{}
{\bfseries}
{:}
{     }
{\thmname{#1} \thmnumber{#2} }

\theoremstyle{remarkstyle}

\begin{document}
	\title{On the Generalized Birth-Death Process and its Linear Versions}
	\author[Pradeep Vishwakarma]{P. Vishwakarma}
	\address{Pradeep Vishwakarma, Department of Mathematics,
		Indian Institute of Technology Bhilai, Durg 491001, INDIA.}
	\email{pradeepv@iitbhilai.ac.in}
	\author{K.k. Kataria}
	\address{Kuldeep Kumar Kataria, Department of Mathematics,
	 Indian Institute of Technology Bhilai, Durg 491001, INDIA.}
	 \email{kuldeepk@iitbhilai.ac.in}
	 
	
	\subjclass[2010]{Primary : 60J27; Secondary: 60J20}
	
	\keywords{Extinction probability; birth-death process; generalized birth-death process; generalized linear birth-death process; generalized linear birth-death process with immigration.}
	\date{\today}
	
	\maketitle
	\begin{abstract}
		In this paper, we consider a generalized birth-death process (GBDP) and examined its linear versions. Using its transition probabilities, we obtain the system of differential equations that governs its state probabilities. The  distribution function of its waiting-time in state $s$ given that it starts in state $s$ is obtained. For a linear version of it, namely, the generalized linear birth-death process (GLBDP), we obtain the probability generating function, mean, variance and the probability of ultimate extinction of population. Also, we obtain the maximum likelihood estimate of one of its parameter. The differential equations that govern the joint cumulant generating functions of the population size with cumulative births and cumulative deaths are derived. In the case of constant birth and death rates in GBDP, the  explicit forms of the state probabilities, joint probability mass functions of population size with cumulative births and cumulative deaths, and their marginal probability mass functions are obtained. It is shown that the Laplace transform of a stochastic integral of GBDP satisfies its Kolmogorov backward equation with certain scaled parameters. Also, the first two moments of the stochastic path integral of GLBDP  are obtained. Later, we consider the  immigration effect in GLBDP for two different cases. An application of a linear version of GBDP and its stochastic path integral to vehicles parking management system is discussed.
	\end{abstract}
	
	\section{Introduction} The birth-death process (BDP) is a continuous-time and discrete state-space Markov process which is used to model the growth of population with time. If $\lambda_n>0$ for all $ n\ge0$ are the birth rates and $\mu_n>0$ for all $ n\ge1$ are the death rates then, in an infinitesimal time interval of length $h$, the probability of a birth equals  $\lambda_n h+o(h)$, probability of a death equals $\mu_n h+o(h)$ and probability of any other event is $o(h)$. Here, $n$ denotes the population size at a given time $t\ge0$. In BDP, the transition rates can depend on the state of the process. If the transition rates in BDP are linear, that is, $\lambda_n=n\lambda$ and $\mu_n=n\mu$ then it is called the linear birth-death process (LBDP).
	The BDP has various applications in the fields of engineering, queuing theory and biological sciences \textit{etc}. For example, it can be used to model the number of customers in queue at a service center or the growth of bacteria over time. However, it has certain limitation, for example, it is not a suitable process to model the situations involving multiple births or multiple deaths in an infinitesimal time interval.
	 	 
Doubleday (1973) introduced and studied a linear birth-death process with multiple births and single death.
Here, we consider a generalized version of the BDP where in an infinitesimal time interval there is a possibility of  multiple but finitely many births or deaths with the assumption that the chance of simultaneous birth and death is negligible. We call this process as the generalized birth-death process (GBDP) and denote it by $\{\mathcal{N}(t)\}_{t\ge0}$. It is shown that the state probabilities $q(n,t)=\mathrm{Pr}\{\mathcal{N}(t)=n\}$, $n\ge0$  of GBDP solve the following system of differential equations:
\begin{equation}\label{dff1}
	\frac{\mathrm{d}}{\mathrm{d}t}q(n,t)=\begin{cases}
		-\sum_{i=1}^{k_1}\lambda_{(0)_i}q(0,t)+\sum_{j=1}^{k_2}\mu_{(j)_j}q(j,t), \ n=0,\vspace{0.1cm}\\
		-\big(\sum_{i=1}^{k_1}\lambda_{(n)_i}+\sum_{j=1}^{k_2}\mu_{(n)_j}\big)q(n,t)\\\hspace{1.8cm}+\sum_{i=1}^{\mathrm{min}\{n,k_1\}}\lambda_{(n-i)_i}q(n-i,t)
		+\sum_{j=1}^{k_2}\mu_{(n+j)_j}q(n+j,t), \ n\ge1
		\end{cases}
\end{equation}
with initial conditions $q(1,0)=1$ and $q(n,0)=0$ for all $n\ne1$. Here, for each $n\ge0$, $\lambda_{(n)_i}$ is the rate of birth of size $i\in\{1,2,\dots,k_1\}$ and for each $n\ge1$, $\mu_{(n)_j}$ is the rate of death of size $j\in\{1,2,\dots,k_2\}$.

If birth and death rates are linear in GBDP, that is, $\lambda_{(n)_i}=n\lambda_i$ and $\mu_{(n)_j}=n\mu_j$ for all $n\ge1$, then we call this process as the generalized linear birth-death process (GLBDP). It is denoted by $\{N(t)\}_{t\ge0}$. The state probabilities $p(n,t)=\mathrm{Pr}\{N(t)=n\}$ of GLBDP are the solution of the following system of differential equations: 
\begin{equation*}
	\frac{\mathrm{d}}{\mathrm{d}t}p(n,t)=\begin{cases}
		\sum_{j=1}^{k_2}j\mu_jp(j,t),\ n=0,\vspace{0.1cm}\\
		-\big(\sum_{i=1}^{k_1}n\lambda_i+\sum_{j=1}^{k_2}n\mu_j\big)p(n,t)\\
		\hspace{1.9cm}+\sum_{i=1}^{k_1}(n-i)\lambda_ip(n-i,t)
		+\sum_{j=1}^{k_2}(n+j)\mu_jp(n+j,t), \ n\ge1,
	\end{cases}		
\end{equation*} 
where $p(n,t)=0$ for all $n<0$.

It is shown that the cumulative distribution function $W_s(t)=\mathrm{Pr}\{T_s\leq t\}$ is given by
\begin{equation*}
	W_s(t)=\exp\Bigg(-\sum_{i=1}^{k_1}\lambda_{(n)_i}+\sum_{j=1}^{k_2}\mu_{(n)_j}\Bigg)t,\ t\ge0,
\end{equation*}
where $T_s$ is the waiting time of GBDP in state $s$ given that it starts in state $s$. On using the distributions of these waiting times, we obtain the maximum likelihood estimate of the parameter $\sum_{i=1}^{k_1}\lambda_{i}+\sum_{j=1}^{k_2}\mu_{j}$ in GLBDP.

In Section~\ref{sec3}, for GLBDP, we obtain some results for the cumulative births $B(t)$, that is, the total number of births by time $t$ and for the cumulative deaths $D(t)$, that is, the total number of deaths by time $t$. We study a particular case of GBDP, where the birth and death rates are constant. We denote this process by $\{N^*(t)\}_{t\ge0}$. Let $B^*(t)$ and $D^*(t)$ be the total number of births and deaths in $\{N^*(t)\}_{t\ge0}$  up to time $t$, respectively. We obtain the explicit forms of probability mass functions (pmfs) of $N^*(t)$, $B^*(t)$ and $D^*(t)$ and their joint distributions.

In Section~\ref{sec5}, we consider a stochastic integral of GBDP and show that its Laplace transform satisfies the backward equation for GBDP with some scaled parameters. We study the joint distribution of GLBDP and its stochastic path integral. A similar study is done for GBDP with constant birth and death rates.

In Section~\ref{sec6}, we study the effect of immigration in GLBDP for two different cases. Let $\nu>0$ is the rate of immigration from outside environment.  First, we consider the immigration effect only if the population vanishes at any time $t\ge0$. In this case, the birth rates are $\lambda_{(0)_i}=\nu$, $\lambda_{(n)_i}=n\lambda_{i},\ n\ge1$ for all $i\in\{1,2,\dots,k_1\}$ and the death rates are $\mu_{(n)_j}=n\mu_j$, $n\ge1$ for all $j\in \{1,2,\dots,k_2\}$. Then, we consider the effect of immigration at every state of the GLBDP. In this case, the birth and death rates are $\lambda_{(n)_i}=\nu+n\lambda_i$ and $\mu_{(n)_j}=n\mu_j$, $n\ge0$, respectively. 

In the last section, we discuss an application of a linear case of GBDP to vehicle parking management system.

\section{Generalized birth-death process}
First, we consider a generalized birth-death process (GBDP) where in an infinitesimal time interval of length $h$ there is a possibility of either finitely many births $i\in\{1,2,\ldots,k_1\}$  with positive rates $\lambda_{(n)_i}$ or finitely many deaths $j\in\{1,2,\ldots,k_2\}$ with positive rates $\mu_{(n)_j}$. It is important to note that the birth and death rates may depend on the number of individuals $n$ present in the population at time $t\ge0$. It is assumed that the chances of simultaneous occurrence of births and deaths in such small intervals are negligible. 

We denote the GBDP by $\{\mathcal{N}(t)\}_{t\ge0}$. With the above assumptions, its transition probabilities are given by
		\begin{multline}\label{gbdp}
		\mathrm{Pr}\{\mathcal{N}(t+h)=n+j|\mathcal{N}(t)=n\}
		=\begin{cases}
				1-\sum_{i=1}^{k_1}\lambda_{(n)_i}h-\sum_{i=1}^{k_2}\mu_{(n)_i}h+o(h),\ j=0, \vspace{0.1cm}\\
				\lambda_{(n)_{j}}h+o(h), \ j=1,2,\ldots,k_1,\vspace{0.1cm}\\
				\mu_{(n)_{-j}}h+o(h), \ j=-1,-2,\ldots,-k_2,\vspace{0.1cm}\\
			o(h),\ \mathrm{otherwise},
		\end{cases}
	\end{multline}
where $o(h)/h\to 0$ as $h\to0$. 

Let $q(n,t)=\mathrm{Pr}\{\mathcal{N}(t)=n\}$, $n\ge0$ be the state probabilities of GBDP. Then, we have
\begin{align*}
  	q(n,t+h)&=\mathrm{Pr}\{\mathcal{N}(t+h)=n|\mathcal{N}(t)=n\}q(n,t)+\sum_{i=1}^{\mathrm{min}\{n,k_1\}}\mathrm{Pr}\{\mathcal{N}(t+h)=n|\mathcal{N}(t)=n-i\}q(n-i,t)\\
  	&\ \ +\sum_{j=1}^{k_2}\mathrm{Pr}\{\mathcal{N}(t+h)=n|\mathcal{N}(t)=n+j\}q(n+j,t).
  \end{align*}  
Using (\ref{gbdp}), we get
\begin{align*}
	q(n,t+h)&=\bigg(1-\sum_{i=1}^{k_1}\lambda_{(n)_i}h-\sum_{j=1}^{k_2}\mu_{(n)_j}h\bigg)q(n,t)+\sum_{i=1}^{\mathrm{min}\{n,k_1\}}\lambda_{(n-i)_i}hq(n-i,t)\\
	&\ \ +\sum_{j=1}^{k_2}\mu_{(n+j)_j}hq(n+j,t)+o(h).
\end{align*}
Equivalently,
\begin{align*}
		\frac{	q(n,t+h)-q(n,t)}{h}&=-\bigg(\sum_{i=1}^{k_1}\lambda_{(n)_i}+\sum_{j=1}^{k_2}\mu_{(n)_j}\bigg)q(n,t)+\sum_{i=1}^{\mathrm{min}\{n,k_1\}}\lambda_{(n-i)_i}q(n-i,t)\\
		&\ \ +\sum_{j=1}^{k_2}\mu_{(n+j)_j}q(n+j,t)+\frac{o(h)}{h}.
\end{align*}
On letting $h\to0$, we get
\begin{equation*}
       \frac{\mathrm{d}}{\mathrm{d}t}q(n,t)=-\bigg(\sum_{i=1}^{k_1}\lambda_{(n)_i}+\sum_{j=1}^{k_2}\mu_{(n)_j}\bigg)q(n,t)+\sum_{i=1}^{\mathrm{min}\{n,k_1\}}\lambda_{(n-i)_i}q(n-i,t)+\sum_{j=1}^{k_2}\mu_{(n+j)_j}q(n+j,t), \ n\ge 1.
\end{equation*}
Similarly, for $n=0$, we have
\begin{equation*}
	q(0,t+h)=\bigg(1-\sum_{i=1}^{k_1}\lambda_{(0)_i}h\bigg)q(0,t)+\sum_{j=1}^{k_2}\mu_{(j)_j}q(j,t)h+o(h).
\end{equation*}
Thus,
\begin{equation*}
	\frac{\mathrm{d}}{\mathrm{d}t}q(0,t)=-\sum_{i=1}^{k_1}\lambda_{(0)_i} q(0,t)+\sum_{j=1}^{k_2}\mu_{(j)_j}q(j,t).
\end{equation*}

Let us assume that there is one progenitor at time $t=0$. Thus, the state probabilities of GBDP solves the following system of differential equations:
\begin{equation}\label{tranprob}
	\frac{\mathrm{d}}{\mathrm{d}t}q(n,t)=-\bigg(\sum_{i=1}^{k_1}\lambda_{(n)_i}+\sum_{j=1}^{k_2}\mu_{(n)_j}\bigg)q(n,t)+\sum_{i=1}^{k_1}\lambda_{(n-i)_i}q(n-i,t)+\sum_{j=1}^{k_2}\mu_{(n+j)_j}q(n+j,t), \ n\ge 0
\end{equation}
with initial conditions $q(1,0)=1$ and $q(n,0)=0$ for all $n\ne1$.
Here, $q(n-i,t)=0$ for all $i>n$ and $\mu_{(0)_j}=0$ for all $j=1,2,\ldots,k_2$.
\begin{remark}
	For $k_1=k_2=1$, GBDP reduces to BDP. If $\mu_{(n)_j}=0$ for all $n$ and for all $j\in \{1,2,\ldots,k_2\}$, then the GBDP reduces to the generalized pure birth process.
\end{remark}
Let $T_s$ be the waiting time of GBDP in the state $s$, that is, $T_s$ is the total time before the process leave the state $s$ given that it start from $s$. 
\begin{proposition}
	Let 
	$
	W_s(t)=\mathrm{Pr}\{T_s\ge t\}.
	$ Then, $W_s(t)$ solves the following linear differential equation:
	\begin{equation}\label{waiting}
		\frac{\mathrm{d}}{\mathrm{d}t}W_s(t)=-\left(\sum_{i=1}^{k_1}\lambda_{(s)_i}+\sum_{j=1}^{k_2}\mu_{(s)_j}\right)W_s(t)
	\end{equation}
	with initial condition $W_s(0)=1.$
\end{proposition}
\begin{proof} Consider an infinitesimal time interval of length $h$ such that $o(h)/h\to 0$ as $h\to0$. On using (\ref{gbdp}), we have
	\begin{align*}
		W_s(t+h)&=W_s(t)\mathrm{Pr}\{\mathcal{N}(t+h)=s|\mathcal{N}(t)=s\}=W_s(t)\bigg(1-\sum_{i=1}^{k_1}\lambda_{(s)_i}h-\sum_{j=1}^{k_2}\mu_{(s)_j}h\bigg)+o(h).
	\end{align*}
	So,
	\begin{equation*}
		\frac{W_s(t+h)-W_s(t)}{h}=-\bigg(\sum_{i=1}^{k_1}\lambda_{(s)_i}+\sum_{j=1}^{k_2}\mu_{(s)_j}\bigg)W_s(t)+\frac{o(h)}{h}.
	\end{equation*}
	On letting $h\to0$, we get the required result.
\end{proof}
\begin{remark}\label{exp}
	On solving (\ref{waiting}), we get
	\begin{equation*}
		W_s(t)=e^{-\left(\sum_{i=1}^{k_1}\lambda_{(s)_i}+\sum_{j=1}^{k_2}\mu_{(s)_j}\right)t},\ t\ge0.
	\end{equation*}
	Thus, the waiting time $T_s$ of GBDP is exponentially distributed with parameter $\sum_{i=1}^{k_1}\lambda_{(s)_i}+\sum_{j=1}^{k_2}\mu_{(s)_j}$.
	For $k_1=k_2=1$, it follows exponential distribution with parameter $\lambda_{s}+\mu_{s}$ (see Karlin and Taylor (1975), p. 133).
\end{remark}

\subsection{A linear case of GBDP}When the birth and death rates are linear, that is, $\lambda_{(n)_{i}}=n\lambda_i$ and $\mu_{(n)_j}=n\mu_j$ for all $i$, $j$ and $n$, we call the GBDP as the generalized linear birth-death process (GLBDP) and denote it by $\{N(t)\}_{t\ge0}$. From (\ref{tranprob}), it follows that the state probabilities $p(n,t)=\mathrm{Pr}\{N(t)=n\}$ of GLBDP solve the following system of differential equations:
\begin{equation}\label{glbdp1}
\frac{\mathrm{d}}{\mathrm{d}t}p(n,t)=
-\left(\sum_{i=1}^{k_1}n\lambda_i+\sum_{j=1}^{k_2}n\mu_j\right)p(n,t)+\sum_{i=1}^{k_1}(n-i)\lambda_ip(n-i,t)+\sum_{j=1}^{k_2}(n+j)\mu_jp(n+j,t), \ n\ge0,	
\end{equation} 
with initial conditions
\begin{equation}\label{initialcon}
	p(n,0)=\begin{cases}
		1, \ n=1,\vspace{0.1cm}\\
		0, \ n\ne 1,
	\end{cases}
\end{equation}
where $p(n-i,t)=0$ for all $i>n$.
\begin{remark}
	If $\mu_j=0$ for all $j\in \{1,2,\ldots,k_2\}$, then the GLBDP reduces to generalized linear pure birth process.
\end{remark}
\begin{remark}
	In the case of GLBDP, the waiting time in state $s$ has the exponential distribution with parameter $s\big(\sum_{i=1}^{k_1}\lambda_i+\sum_{j=1}^{k_2}\mu_j\big)$.
\end{remark}

Next, we obtain the system of differential equations that governs the probability generating function (pgf) of GLBDP.
\begin{proposition}\label{thm1}
Let $H(u,t)=\mathbb{E}(u^{N(t)}),\ |u|\leq1$ be the pgf of GLBDP. It solves the following partial differential equation:
   \begin{equation}\label{pgfequ}
    	\frac{\partial}{\partial t}H(u,t)=\left(\sum_{i=1}^{k_1}\lambda_iu(u^i-1)+\sum_{j=1}^{k_2}\mu_ju(u^{-j}-1)\right)\frac{\partial}{\partial u}H(u,t),\ t\ge0,
\end{equation}
with initial condition $H(u,0)=u$.
\end{proposition}

\begin{proof}
By using (\ref{initialcon}), we have
$
	H(u,0)=\sum_{n=0}^{\infty}u^np(n,0)=u.
$
   On multiplying $u^n$ on both sides of (\ref{glbdp1}) and by adjusting the terms, we get   
\begin{align}
	\frac{\partial}{\partial t}u^np(n,t)&=-\bigg(\sum_{i=1}^{k_1}\lambda_i+\sum_{j=1}^{k_2}\mu_j\bigg)u\frac{\partial}{\partial u}u^np(n,t)+\sum_{i=1}^{k_1}\lambda_i u^{i+1}\frac{\partial}{\partial u}u^{n-i} p(n-i,t)\nonumber\\
		&\ \ +\sum_{j=1}^{k_2}\mu_ju^{-j+1}\frac{\partial}{\partial u}u^{n+j}p(n+j,t).\label{*}
\end{align}
Summing over $n=-k_2,-k_2+1,\ldots$ on both sides of (\ref{*}) gives the required result. Here, the change of  sum and derivative is justified because
\begin{align*}
	\left|\sum_{n=0}^{\infty}\frac{\partial}{\partial u}u^np(n,t)\right|&=\left|\sum_{n=0}^{\infty}nu^{n-1}p(n,t)\right|\\
	&\leq\sum_{n=0}^{\infty}np(n,t)=\mathbb{E}(N(t))<\infty.
\end{align*}
This completes the proof.
\end{proof}
\begin{remark}
	For $k_1=k_2=1$, the GLBDP reduces to LBDP. Its pgf  solves (see Bailey (1964)) 
	\begin{equation*}	
		\frac{\partial}{\partial t}{H}(u,t)=(\lambda_1 u-\mu_1)(u-1)\frac{\partial}{\partial u}{H}(u,t)
	\end{equation*}
	with ${H}(u,0)=u$.
\end{remark}

\begin{remark}
		The Laplace transform $\tilde{H}(u,z)=\int_{0}^{\infty}e^{-zt}H(u,t)\,\mathrm{d}t$ of the pgf of GLBDP
		is the solution of following differential equation:
		\begin{equation*}
			-u+z	\tilde{H}(u,z)=\bigg(\sum_{i=1}^{k_1}\lambda_iu(u^i-1)+\sum_{j=1}^{k_2}\mu_ju(u^{-j}-1)\bigg)\frac{\partial}{\partial u}	\tilde{H}(u,z).
		\end{equation*}	 
\end{remark}
Next, we discuss the solution of (\ref{pgfequ}).
First, we rewrite (\ref{pgfequ}) as follows:
\begin{equation}\label{pgf1}
	\frac{\partial}{\partial t}H(u,t)=\phi(u)\frac{\partial}{\partial u}H(u,t),
\end{equation}
where $\phi(u)=\sum_{i=1}^{k_1}\lambda_iu(u^i-1)+\sum_{j=1}^{k_2}\mu_ju(u^{-j}-1).$
We apply the method of characteristics to solve (\ref{pgf1}), and thus we get the following subsidiary equations: 

\begin{equation}\label{sub1}
	\mathrm{d}t=-\frac{u^{k_2-1}}{\psi(u)}\mathrm{d}u
\end{equation}
and 
\begin{equation}\label{sub2}
	\frac{\partial}{\partial t}H(u,t)=0,
\end{equation}
where $\psi(u)=u^{k_2-1}\phi(u)$.
Moreover, we assume that the roots $r_i,\ i=1,2,\ldots,k_1+k_2$ of $\psi(u)$ are distinct which can be ensured by a small perturbation in $\lambda_i$ and $\mu_j$.
 So, (\ref{sub1}) can be rewritten as
\begin{equation}\label{2.9}
	\mathrm{d}t=-\left(\sum_{i=1}^{k_1+k_2}\frac{c_i}{(u-r_i)}\right)\mathrm{d}u,
\end{equation}
where $c_i=\lambda_{k_1}^{-1}{r_i^{k_2-1}}/{\prod_{j\ne i}(r_i-r_j)}.$

On solving (\ref{sub2}) and (\ref{2.9}), we get $	H(u,t)=\beta_1$ and	$e^{-t}g(u)= \beta_2$,
respectively. Here, $\beta_1$ and $\beta_2$ are  real constants and 
\begin{equation}\label{g(u)}
	g(u)=\prod_{i=1}^{k_1+k_2}(u-r_i)^{-c_i}.
\end{equation}
Thus, the arbitrary solution of  (\ref{pgfequ}) is obtained in the following form:
\begin{equation}\label{PGF}
	H(u,t)=f\left(e^{-t}g(u)\right),
\end{equation}
where $f$ is some real-valued function which can be determined by using the given initial condition, that is, $	H(u,0)=u=f\left(g(u)\right)$. Note that there always exist a $\delta>0$ such that $|r_i|>\delta$ and $\psi(u)$ has same sign for all $u\in [0,\delta]$. So, $g(u)$ is a monotone function in $[0,\delta]$ and its inverse exist.
Hence, (\ref{PGF}) holds with $f(u)=g^{-1}(u)$ for all $u\in [0,\delta]$ and by using the analytic continuation it can be extended to all $u$ such that $|u|\leq1$.
\begin{remark}
		Let $\left(a_k:k\in \{0,1,2,\ldots\}\right)$ be the initial distribution of GLBDP. Then,  its pgf has the following form: 
	\begin{equation*}
		H(u,t)=\sum_{k=0}^{\infty}a_k\left(f\left(e^{-t}\prod_{i=1}^{k_1+k_2}(u-r_i)^{-c_i}\right)\right)^{k}.
	\end{equation*}	
	In particular, if $N(0)=n_0>1$ then
	\begin{equation*}
		H(u,t)=\left(f\left(e^{-t}\prod_{i=1}^{k_1+k_2}(u-r_i)^{-c_i}\right)\right)^{n_0}.
	\end{equation*}
	Its proof follows along the similar lines to that of (\ref{PGF}).
\end{remark}
Next, we obtain the probability of ultimate extinction of GLBDP.

Let $p(0)=\lim_{t\to\infty}p(0,t)$ be the probability of ultimate extinction in GLBDP. We use the inverse function representation of $H(u,t)$ to approximate the extinction probability $ p(0,t)$.
Let ${g}'(u)=\frac{\mathrm{d}g(u)}{\mathrm{d}u}$, where $g$ is given in (\ref{g(u)}).
If $\epsilon$ is the least positive root of $\psi(u)$ then $g(\cdot)$ is a monotone function in $[0,\epsilon]$ and $g(\epsilon)=0$. So,
$
	f(u)=g^{-1}(u)$ for all $u\in[0,\epsilon] $ {and}$\ f(0)=\epsilon.
$
As
\begin{align*}
	f(u)&=\epsilon+\int_{0}^{u}f'(x)\,\mathrm{d}x\\
	&=\epsilon+\int_{0}^{u}\frac{\mathrm{d}x}{{g}'(f(x))}\ \ \ \text{for}\ u\in [0,\delta],
\end{align*}
where $\delta$ is some positive real number. So, we have
\begin{align*}
	H(u,t)=\epsilon+\int_{0}^{e^{-t}g(u)}\frac{\mathrm{d}x}{{g}'(f(x))}.
\end{align*}
On taking $u=0$, we get
\begin{equation*}
	p(0,t)=H(0,t)=\epsilon+\int_{0}^{e^{-t}g(0)}\frac{\mathrm{d}x}{{g}'(f(x))}.
\end{equation*}
Thus, if the process run for infinitely long time then the probability of ultimate extinction converges exponentially to $\epsilon$, that is,
\begin{equation}\label{extprob}
	p(0)=\lim_{t\to\infty}p(0,t)=\epsilon.
\end{equation}
\begin{remark}
	In the case of LBDP, the above results reduce to the results given in Bailey (1964).
	
	 For $k_1=k_2=1$, $ \lambda_1=\lambda$ and $\mu_1=\mu$ such that $\lambda\ne\mu$, we get the two distinct roots of $\psi(u)$, that is, $r_1=\mu/\lambda$ and $r_2=1$. So, $c_1=1/(\mu-\lambda)$ and $c_2=1/(\lambda-\mu)$.
From (\ref{g(u)}), we have
\begin{equation*}
	g(u)=\left(\frac{u-\mu/\lambda}{u-1}\right)^{1/(\lambda-\mu)}.
\end{equation*}
Thus, 
\begin{equation*}
	f(u)=g^{-1}(u)=\frac{\mu-\lambda u^{\lambda-\mu}}{\lambda(1-u^{\lambda-\mu})}.
\end{equation*}
So, from (\ref{PGF}), the pgf of LBDP is 
\begin{equation*}
	H(u,t)=\frac{\mu(u-1)-e^{-t(\lambda-\mu)}(\lambda u-\mu)}{\lambda(u-1)-e^{-t(\lambda-\mu)}(\lambda u-\mu)}.
\end{equation*}
Its extinction probability is given by (see Bailey (1964), p. 95)
\begin{equation*}
	p(0,t)=H(0,t)=\frac{\mu-\mu e^{-t(\lambda-\mu)}}{\lambda-\mu e^{-t(\lambda-\mu)}}
\end{equation*}
and its state probabilities are (see Bailey (1964), p. 95)
\begin{equation*}
	p(n,t)=(\lambda-\mu)^2e^{-t(\lambda-\mu)}\frac{(\lambda-\lambda e^{-t(\lambda-\mu)})^{n-1}}{(\lambda-\mu e^{-t(\lambda-\mu)})^{n+1}},\ n\ge1.
\end{equation*}
From (\ref{extprob}), the probability of ultimate extinction $p(0)$ of LBDP is given by (see Bailey (1964), p. 96)
	\begin{equation*}
		p(0)=\begin{cases}
			\frac{\mu}{\lambda},\ \lambda>\mu,\vspace{0.1cm}\\
			1,\ \lambda<\mu.
		\end{cases}
	\end{equation*}
\end{remark}
\begin{remark}
	The moment generating function (mgf) $M(u,t)\coloneqq H(e^u,t)$ of GLBDP solves 
	\begin{equation*}
		\frac{\partial}{\partial t}M(u,t)=\left(\sum_{i=1}^{k_1}\lambda_i(e^{iu}-1)+\sum_{j=1}^{k_2}\mu_j(e^{-ju}-1)\right)\frac{\partial}{\partial u}M(u,t),
	\end{equation*}
with $M(u,0)=e^{u}$.
\end{remark}

Next, we obtain the mean and variance of GLBDP.	
\begin{theorem}
 Let $\eta=\sum_{i=1}^{k_1}i\lambda_i-\sum_{j=1}^{k_2}j\mu_j$.	For $t\ge0$, the mean $m(t)=\mathbb{E}(N(t))$ and the variance $v(t)=\mathrm{Var}(N(t))$ of GLBDP are given by
	\begin{equation}\label{meanvar}
		m(t)=e^{\eta t}\ \ \text{and}\ \ 	v(t)=\frac{\sum_{i=1}^{k_1}i^2\lambda_i+\sum_{j=1}^{k_2}j^2\mu_j}{\eta}\left(e^{2\eta t}-e^{\eta t}\right),
	\end{equation}
	respectively.
\end{theorem}
\begin{proof}
		From (\ref{initialcon}), we get
 $m(0)=\sum_{n=1}^{\infty}np(n,0)=1$.
	On multiplying $n$ on both sides of (\ref{glbdp1}) and substituting $n(n-i)=(n-i)^2+i(n-i)$ and $n(n+j)=(n+j)^2-j(n+j)$ yields	
	\begin{multline*}
		\frac{\mathrm{d}}{\mathrm{d}t}m(t)=	-\left(\sum_{i=1}^{k_1}\lambda_i+\sum_{j=1}^{k_2}\mu_j\right) n^2p(n,t)+\sum_{i=1}^{k_1}\lambda_i(n-i)^2p(n-i,t)+\sum_{j=1}^{k_2}\mu_j(n+j)^2p(n+j,t)\\
		+\sum_{i=1}^{k_1}\lambda_ii(n-i)p(n-i,t)-\sum_{j=1}^{k_2}\mu_jj(n+j)p(n+j,t).
	\end{multline*}
	Summing over $n=-k_2,-k_2+1,\ldots$ gives
	\begin{align}\label{2.15}
		\frac{\mathrm{d}}{\mathrm{d}t}m(t)=\left(\sum_{i=1}^{k_1}i\lambda_i-\sum_{j=1}^{k_2}j\mu_j\right)m(t).
	\end{align}
	On solving (\ref{2.15}) with initial condition $m(0)=1$, we get $m(t)=\exp{\big(\sum_{i=1}^{k_1}i\lambda_i-\sum_{j=1}^{k_2}j\mu_j\big) t}$.
	
	Let $m_2(t)=\sum_{n=0}^{\infty}n^2p(n,t)$. In order to obtain the variance, we multiply $n^2$ on both sides of (\ref{glbdp1}) and substitute $n^2(n-i)=(n-i)^3+2i(n-i)^2+i^2(n-i)$ and $n^2(n+j)=(n+j)^3-2j(n+j)^2+j^2(n+j)$, to get
	\begin{align*}
		\frac{\mathrm{d}}{\mathrm{d}t}m_2(t)&=-\left(\sum_{i=1}^{k_1}\lambda_i+\sum_{j=1}^{k_2}\mu_j\right) n^3p(n,t)+\sum_{i=1}^{k_1}\lambda_i(n-i)^3p(n-i,t)+2\sum_{i=1}^{k_1}\lambda_ii(n-i)^2p(n-i,t)\\
		&\quad
		+\sum_{i=1}^{k_1}\lambda_ii^2(n-i)p(n-i,t)+\sum_{j=1}^{k_2}\mu_j(n+j)^3p(n+j,t)-2\sum_{j=1}^{k_2}\mu_jj(n+j)^2p(n+j,t)\\
		&\quad+\sum_{j=1}^{k_2}\mu_jj^2(n+j)p(n+j,t)\\
		&=2\sum_{i=1}^{k_1}\lambda_ii(n-i)^2p(n-i,t)
		+\sum_{i=1}^{k_1}\lambda_ii^2(n-i)p(n-i,t)-2\sum_{j=1}^{k_2}\mu_jj(n+j)^2p(n+j,t)\\
		&\quad+\sum_{j=1}^{k_2}\mu_jj^2(n+j)p(n+j,t).
	\end{align*}
	 By summing over $n=-k_2,-k_2+1,\ldots$, we get
	\begin{equation}\label{Var}
		\frac{\mathrm{d}}{\mathrm{d}t}m_2(t)=	2\left(\sum_{i=1}^{k_1}i\lambda_i-\sum_{j=1}^{k_2}j\mu_j\right)m_2(t)+\left(\sum_{i=1}^{k_1}i^2\lambda_i+\sum_{j=1}^{k_2}j^2\mu_j\right)m(t).
	\end{equation}
	On solving (\ref{Var}) with initial condition $m_2(0)=1$, we get
	\begin{equation*}
		m_2(t)=\left(1+\frac{\sum_{i=1}^{k_1}i^2\lambda_i+\sum_{j=1}^{k_2}j^2\mu_j}{\sum_{i=1}^{k_1}i\lambda_i-\sum_{j=1}^{k_2}j\mu_j}\right)m(t)^2-\frac{\sum_{i=1}^{k_1}i^2\lambda_i+\sum_{j=1}^{k_2}j^2\mu_j}{\sum_{i=1}^{k_1}i\lambda_i-\sum_{j=1}^{k_2}j\mu_j}m(t).
	\end{equation*}
	Hence,
	\begin{align*}
		v(t)&=m_2(t)-m(t)^2=\left(\frac{\sum_{i=1}^{k_1}i^2\lambda_i+\sum_{j=1}^{k_2}j^2\mu_j}{\sum_{i=1}^{k_1}i\lambda_i-\sum_{j=1}^{k_2}j\mu_j}\right)(m(t)^2-m(t)).
	\end{align*}
 This completes the proof.
\end{proof}
{\subsubsection{Parameter estimation in GLBDP}\label{sec4}
	Moran (1951) and Moran (1953) used the inter-event time to estimate the parameters in evolutive processes. Doubleday (1973) used similar technique to obtain the maximum likelihood estimation (MLE) of the parameters in multiple birth-death process.	
	Here, we use the same method to obtain the MLE of the parameter  $		\Lambda=\sum_{i=1}^{k_1}\lambda_i+\sum_{i=1}^{k_2}\mu_i$ in GLBDP.
	
	Let $\mathcal{E}$ be the total number of transitions made by $N(t)$ upto some fixed time $\mathcal{T}$. Suppose $t_1,t_2,\ldots,t_\mathcal{E}$ are epochs  when these transitions have occurred. We define 
	$
	\tau_k=t_k-t_{k-1}\ \text{for all}\ 1\leq k\leq \mathcal{E}\ \text{with}\ t_0=0,
	$
	where $\tau_1,\tau_2,\ldots,\tau_\mathcal{E}$ denote the inter-event times.
	 Note that $\tau_k$'s are independent and exponentially distributed with parameter $\Lambda N(t_{k-1})$, $k=1,2,\ldots,\mathcal{E}$.
	The likelihood function is given by 
	\begin{equation}\label{jointd}
		L(\Lambda;\tau_1,\tau_2,\ldots,\tau_\mathcal{E})=\Lambda^n\prod_{k=1}^{\mathcal{E}}N(t_{k-1})e^{-\Lambda\sum_{k=1}^{\mathcal{E}}N(t_{k-1})\tau_k}.
	\end{equation}
	On taking the derivative of log-likelihood function $\ln	L(\Lambda;\tau_1,\tau_2,\ldots,\tau_\mathcal{E})$ with respect to $\Lambda$, we get
	\begin{equation}\label{mle}
		\frac{\mathrm{d}}{\mathrm{d}\Lambda}\ln	L(\Lambda;\tau_1,\tau_2,\ldots,\tau_\mathcal{E})=\frac{\mathcal{E}}{\Lambda}-\sum_{k=1}^{\mathcal{E}}N(t_{k-1})\tau_k.
	\end{equation}
	On equating (\ref{mle}) to zero, we get the MLE of $\Lambda$ as follows:
	\begin{equation*}
		\hat{\Lambda}_{mle}=\frac{\mathcal{E}}{\sum_{k=1}^{\mathcal{E}}N(t_{k-1})\tau_k}.
	\end{equation*} 
	Note that $2\Lambda N(t_{k-1})\tau_k $ follows $\chi^2$ distribution  with two degrees of freedom. So,
	\begin{equation*}
		2\Lambda\sum_{k=1}^{\mathcal{E}}N(t_{k-1})\tau_k=	2\Lambda\int_{0}^{\mathcal{T}}N(s)\,\mathrm{d}s
	\end{equation*}
	has $\chi^2$ distribution with $2\mathcal{E}$ degrees of freedom. Thus,
	$
	\mathrm{E}\big(2\Lambda\int_{0}^{\mathcal{T}}N(s)\,\mathrm{d}s\big)=2\mathcal{E}.
	$
	Hence,
	$
	\tilde{\Lambda}={\mathcal{E}}^{-1}\int_{0}^{\mathcal{T}}N(s)\,\mathrm{d}s
	$
	is an unbiased estimator of $1/\Lambda$ whose distribution is known. From (\ref{jointd}), it follows by using the Fisher factorization theorem that
	$
	\tilde{\Lambda}={\mathcal{E}}^{-1}\sum_{k=1}^{\mathcal{E}}N(t_{k-1})\tau_k
	$
	is a sufficient statistics for $\Lambda$. Indeed, it is a minimal sufficient statistics for $\Lambda$.
\subsection{GBDP with constant birth and death rates}\label{subsec1}
Here, we obtain the explicit form of the state probabilities of GBDP in the case of constant birth and death rates, that is, $\lambda_{(n)_i}=\lambda_i>0$  for all $n\ge0$ and $\mu_{(n)_j}=\mu_j>0$ for all $n\ge1$. In this case, we denote the process by $\{N^*(t)\}_{t\ge0}$. 

 From (\ref{tranprob}), the state probabilities  $q^*(n,t)=\mathrm{Pr}\{N^*(t)=n\}$ solve the following system of differential equations:
\begin{equation}\label{ctranprob}
	\frac{\mathrm{d}}{\mathrm{d}t}q^*(n,t)=-\bigg(\sum_{i=1}^{k_1}\lambda_{i}+\sum_{j=1}^{k_2}\mu_{j}\bigg)q^*(n,t)+\sum_{i=1}^{k_1}\lambda_{i}q^*(n-i,t)+\sum_{j=1}^{k_2}\mu_{j}q^*(n+j,t), \ n\ge 0
\end{equation}
with initial conditions $q^*(1,0)=1$ and $q^*(n,0)=0$ for all $n\ne1$.
\begin{remark}
	If there is no progenitor at $t=0$ and $\lambda_{(n)_i}=\lambda_i>0$, $\mu_{(n)_j}=0$ for all $i$, $j$ and $n\ge0$ then the GBDP reduces to the generalized counting process (GCP) (see Crescenzo et. al. (2016)).
\end{remark}
\begin{remark}
	The pgf $H^*(u,t)=\mathbb{E}\left(u^{N^*(t)}\right),\ |u|\leq1$ is the solution of the following differential equation:
	\begin{equation*}
		\frac{\partial}{\partial t}H^*(u,t)=\left(\sum_{i=1}^{k_1}\lambda_i\left(u^i-1\right)+\sum_{j=1}^{k_2}\mu_j\left(u^{-j}-1\right)\right)H^*(u,t)
	\end{equation*}
	with $H^*(u,0)=u$. It is given by
	\begin{equation}\label{conspgfsol}
		H^*(u,t)=u \exp\left(\sum_{i=1}^{k_1}\lambda_i\left(u^i-1\right)+\sum_{j=1}^{k_2}\mu_j\left(u^{-j}-1\right)\right)t.
	\end{equation}
\end{remark}

\begin{theorem}
 The state probabilities $q^*(n,t)$ of the process $\{N^*(t)\}_{t\ge0}$ are given by
\begin{equation*}
	q^*(n,t)=\sum_{\Omega^*(k_1,k_2,n)}\prod_{i=1}^{k_1}\prod_{j=1}^{k_2}\frac{\lambda_i^{x_i}\mu_j^{y_j}t^{x_i+y_j}}{x_i!y_j!}e^{-\Lambda t},\ n\ge0,
\end{equation*}
where 
	$
	\Omega^*(k_1,k_2,n)=\bigg\{(x_1,x_2,\ldots,x_{k_1},y_1,y_2,\ldots,y_{k_2})\in \mathbb{N}_0^{k_1+k_2}:\sum_{i=1}^{k_1}ix_i-\sum_{j=1}^{k_2}jy_j=n-1\bigg\}
$ and $\Lambda=\left(\sum_{i=1}^{k_1}\lambda_i+\sum_{j=1}^{k_2}\mu_j\right)$.

\end{theorem}
\begin{proof}	
	From (\ref{conspgfsol}), we have
	\begin{align}
		H^*(u,t)&=ue^{-\Lambda t}\sum_{l=0}^{\infty}\bigg(\sum_{i=1}^{k_1}\lambda_iu^i+\sum_{j=1}^{k_2}\mu_ju^{-j}\bigg)^l\frac{t^l}{l!}\nonumber\\
		&=e^{-\Lambda t}\sum_{l=0}^{\infty}\sum_{S(k_1,k_2,l)}\prod_{i=1}^{k_1}\prod_{j=1}^{k_2}\frac{\lambda_i^{x_i}\mu_j^{y_j}t^{x_i+y_j}}{x_i!y_j!}u^{ix_i-jy_j+1},\label{pgfexp}
	\end{align}	
	where
	\begin{equation*}
	S(k_1,k_2,l)=\bigg\{(x_1,x_2,\ldots,x_{k_1},y_1,y_2,\ldots,y_{k_2}): x_i, y_j\in\{0,1,\ldots,l\}\  \text{and}\ \sum_{i=1}^{k_1}x_i+\sum_{j=1}^{k_2}y_j=l\bigg\}.
	\end{equation*}
 On rearranging the terms and extracting the coefficients of $u^n$ in (\ref{pgfexp}), we get the required result.
\end{proof}

\begin{theorem}
	For $t\ge0$, the mean  and the variance of $N^*(t)$ are given by 
	\begin{equation*}
		\mathbb{E}(N^*(t))=1+\eta t\ \ \text{and}\ \ \mathrm{Var}(N^*(t))=\Bigg(\sum_{i=1}^{k_1}i^2\lambda_i+\sum_{j=1}^{k_2}j^2\mu_j\Bigg)t,
	\end{equation*}
	respectively.
\end{theorem}
\begin{proof}
	On taking the derivatives with respect to $u$ on both sides of (\ref{conspgfsol}), we get
	\begin{equation*}
		\frac{\partial}{\partial u}H^*(u,t)=e^{-\Lambda t}e^{\left(\sum_{i=1}^{k_1}\lambda_iu^i+\sum_{j=1}^{k_2}\mu_ju^{-j}\right)t}\left(1+\bigg(\sum_{i=1}^{k_1}i\lambda_iu^{i}-\sum_{j=1}^{k_2}j\mu_ju^{-j}\bigg)t\right)
	\end{equation*}
	and
	\begin{align*}
		\frac{\partial^2}{\partial u^2}H^*(u,t)&=e^{-\Lambda t}e^{\left(\sum_{i=1}^{k_1}\lambda_iu^i+\sum_{j=1}^{k_2}\mu_ju^{-j}\right)t}\left(\sum_{i=1}^{k_1}i\lambda_iu^{i-1}-\sum_{j=1}^{k_2}j\mu_ju^{-j-1}\right)t\\
		&\ \ +\left(\sum_{i=1}^{k_1}i(i-1)\lambda_iu^{i-2}+\sum_{j=1}^{k_2}j(j+1)\mu_ju^{-j-2}\right)tH^*(u,t)\\
		&\ \ +\left(\sum_{i=1}^{k_1}i\lambda_iu^{i-1}-\sum_{j=1}^{k_2}j\mu_ju^{-j-1}\right)t\frac{\partial}{\partial u}H^*(u,t).
	\end{align*}
	So,
	$
		\mathbb{E}(N^*(t))=\frac{\partial}{\partial u}H^*(u,t)\bigg|_{u=1}=1+\eta t
	$
	and
	\begin{equation*}
		\mathrm{Var}(N^*(t))=\frac{\partial^2}{\partial u^2}H^*(u,t)\bigg|_{u=1}+\mathbb{E}(N^*(t))-(\mathbb{E}(N^*(t)))^2=\left(\sum_{i=1}^{k_1}i^2\lambda_i+\sum_{j=1}^{k_2}j^2\mu_j\right)t.
	\end{equation*}
	This completes the proof.
\end{proof}

\section{Cumulative births and deaths in GLBDP}\label{sec3}
Let $B(t)$ denotes the total number of individuals born by time $t$. 

\begin{proposition}\label{prop3.1}
	The joint pmf $	p(b,n,t)=\mathrm{Pr}\{B(t)=b,{N}(t)=n\}$ 
	solves the following system of differential equations:
	\begin{align}\label{cumdis}
		\frac{\mathrm{d}}{\mathrm{d}t}p(b,n,t)&=-\left(\sum_{i=1}^{k_1}\lambda_i+\sum_{j=1}^{k_2}\mu_j\right)np(b,n,t)+\sum_{i=1}^{k_1}(n-i)\lambda_ip(b-i,n-i,t)\nonumber\\
		&\ \ +
		\sum_{j=1}^{k_2}(n+j)\mu_jp(b,n+j,t),\ b\ge1,\,n\ge0,
	\end{align}
with initial conditions
\begin{equation}\label{cumini}
	p(b,n,0)=\begin{cases}
		1,\ b=n=1,\\
		0,\ \text{otherwise},
	\end{cases}
\end{equation} 
 where $p(b,n,t)=0$ for all $b<0$ or $n<0$.
\end{proposition}
\begin{proof}
	In an infinitesimal time
	 interval of length $h$ such that $o(h)/h\to0$ as $h\to0$, we have
	\begin{multline*}
		\mathrm{Pr}\{B(t+h)=b+j,{N}(t+h)=n+k|\,B(t)=b,{N}(t)=n\}\\
		=\begin{cases}
			1-\left(\sum_{i=1}^{k_1}\lambda_i+\sum_{i=1}^{k_2}\mu_i\right)nh+o(h),\ j=k=0,\vspace{0.1cm}\\
			n\lambda_jh+o(h),\ j=k=1,2,\ldots,k_1,\vspace{0.1cm}\\
			n\mu_{-k}h+o(h),\ j=0,\,-k=1,2,\ldots,k_2,\vspace{0.1cm}\\
			o(h),\ \mathrm{otherwise}.
		\end{cases}
	\end{multline*} 
 So,
	\begin{align*}
		p(b,n,t+h)&=\left(1-\Bigg(\sum_{i=1}^{k_1}\lambda_i+\sum_{j=1}^{k_2}\mu_j\Bigg)nh\right)p(b,n,t)+\sum_{i=1}^{k_1}(n-i)\lambda_ihp(b-i,n-i,t)\\
		&\ \ +\sum_{j=1}^{k_2}(n+j)\mu_jhp(b,n+j,t)+o(h).
	\end{align*}
Equivalently,
\begin{align*}
	\frac{p(b,n,t+h)-p(b,n,t)}{h}&=-\left(\sum_{i=1}^{k_1}\lambda_i+\sum_{j=1}^{k_2}\mu_j\right)np(b,n,t)+\sum_{i=1}^{k_1}(n-i)\lambda_ip(b-i,n-i,t)\\
	&\ \ +\sum_{j=1}^{k_2}(n+j)\mu_jp(b,n+j,t)+\frac{o(h)}{h}.
\end{align*}
On letting $h\to0$, we get the required result. 
\end{proof}
\begin{proposition}\label{prop3.2}
	The pgf  $G(u,v,t)=\mathbb{E}\left(u^{B(t)}v^{{N}(t)}\right)$,  $|u|\leq1$, $|v|\leq1$
solves
	\begin{equation}\label{cumpgf}
			\frac{\partial}{\partial t}G(u,v,t)=\left(\sum_{i=1}^{k_1}\lambda_iv((uv)^i-1)+\sum_{j=1}^{k_2}\mu_jv(v^{-j}-1)\right)\frac{\partial}{\partial v}G(u,v,t)
	\end{equation}
	with $G(u,v,0)=uv.$
\end{proposition}
\begin{proof} Using (\ref{cumini}), we get $G(u,v,0)=\sum_{b=0}^{\infty}\sum_{n=0}^{\infty}p(b,n,0)u^bv^n=uv$. On multiplying with $u^bv^n$ on both sides of (\ref{cumdis}) and by adjusting the terms, we get
\begin{align*}
	\frac{\partial}{\partial t}G(u,v,t)&=-\left(\sum_{i=1}^{k_1}\lambda_i+\sum_{j=1}^{k_2}\mu_j\right)v\frac{\partial}{\partial v}u^bv^np(b,n,t)+\sum_{i=1}^{k_1}\lambda_iu^iv^{i+1}\frac{\partial}{\partial v}u^{b-i}v^{n-i}p(b-i,n-i,t)\\
	&\ \ +\sum_{j=1}^{k_2}\mu_jv^{-j+1}\frac{\partial}{\partial v}u^bv^{n+j}p(b,n+j,t).
\end{align*}
On summing over $b=0,1,2,\dots$ and $n=-k_2,-k_2+1,\dots$, we get the required result.
\end{proof}
Let $\theta_u\leq0$ and $\theta_v\leq0$. On taking $u=e^{\theta_u}$ and $v=e^{\theta_v}$ in (\ref{cumpgf}), we get the cumulant generating function (cgf) $K(\theta_u,\theta_v,t)=\ln G(e^{\theta_u},e^{\theta_v},t)$ as a solution of the following differential equation:
\
	\begin{equation}\label{cumpgf1}
	\frac{\partial}{\partial t}K(\theta_u,\theta_v,t)=\left(\sum_{i=1}^{k_1}\lambda_i\left(e^{i(\theta_u+\theta_v)}-1\right)+\sum_{j=1}^{k_2}\mu_j\left(e^{-j\theta_v}-1\right)\right)\frac{\partial}{\partial \theta_v}K(\theta_u,\theta_v,t)	
	\end{equation} 
	with  $K(\theta_u,\theta_v,0)={\theta_u+\theta_v}.$
	
In (\ref{cumpgf1}), we use the following expanded form of the cgf (see Kendall (1948), Eq. (35)):
\begin{equation}\label{cgfequ}
	K(\theta_u,\theta_v,t)= \theta_u\mathbb{E}(B(t))+\theta_v\mathbb{E}(N(t))+\theta_u^2/2\mathrm{Var}(B(t))+\theta_v^2/2\mathrm{Var}(N(t))+\theta_u\theta_v\mathrm{Cov}(B(t),N(t))+\dots.
\end{equation}
to obtain 
\begin{align}
	\frac{\mathrm{d}}{\mathrm{d}t}\bigg(\theta_u\mathbb{E}(B(t))&+\theta_v\mathbb{E}({N}(t))+\frac{\theta_u^2}{2}\mathrm{Var}(B(t))+\frac{\theta_v^2}{2}\mathrm{Var}({N}(t))+\theta_u\theta_v\mathrm{Cov}(B(t),{N}(t))+\cdots\bigg)\nonumber\\
	&=\sum_{l=1}^{\infty}\left(\sum_{i=1}^{k_1}\lambda_i\frac{((\theta_u+\theta_v)i)^l}{l!}+\sum_{j=1}^{k_2}\mu_j\frac{(-\theta_vj)^l}{l!}\right)\big(\mathbb{E}({N}(t))+\theta_v\mathrm{Var}({N}(t))\nonumber\\
	&\ \ +\theta_u\mathrm{Cov}(B(t),{N}(t))+\cdots\big).\label{comp}
\end{align}

\begin{proposition}\label{prop3.3}
	For $t\ge0$, we have\\
	\noindent (i) $
	\mathbb{E}({N}(t))=\begin{cases}
		e^{\eta t},\ \eta\ne0,\vspace{0.1cm}\\
		1,\ \eta=0,
	\end{cases}
	$\\
	\noindent (ii) $ \mathrm{Var}({N}(t))=\begin{cases}
		\frac{\zeta(e^{2\eta t}-e^{\eta t})}{\eta},\ \eta\ne0,\vspace{0.1cm}\\
		\zeta t,\ \eta=0,
	\end{cases}$\\
	\noindent (iii) $
	\mathrm{Cov}(B(t),{N}(t))=\begin{cases}
		\frac{\zeta\sum_{i=1}^{k_1}i\lambda_i}{\eta^2}\left(e^{2\eta t}-e^{\eta t}\right)-\frac{\xi}{\eta}te^{\eta t},\ \eta\ne0,\vspace{0.1cm}\\
		\sum_{i=1}^{k_1}i^2\lambda_it+\zeta\sum_{i=1}^{k_1}i\lambda_i\frac{t^2}{2},\ \eta=0,
	\end{cases}$\\
	
	\noindent (iv) $
	\mathbb{E}(B(t))=\begin{cases}
		\frac{\sum_{i=1}^{k_1}i\lambda_i}{\eta}e^{\eta t}-\frac{\sum_{j=1}^{k_2}j\mu_j}{\eta},\ \eta\ne0,\vspace{0.1cm}\\
		\sum_{i=1}^{k_1}i\lambda_it+1,\ \eta=0,
	\end{cases}$\\
	
	\noindent (v)
	$
	\mathrm{Var}(B(t))=\begin{cases}
		\frac{\sum_{i=1}^{k_1}i^2\lambda_i}{\eta}(e^{\eta t}-1)+2\sum_{i=1}^{k_1}i\lambda_i\bigg(\frac{\zeta \sum_{i=1}^{k_1}i\lambda_i}{2\eta^3}\left(e^{\eta t}-1\right)^2-\frac{\xi}{\eta^3}\left(\eta t e^{\eta t}-e^{\eta t}+1\right)\bigg),\ \eta\ne0,\vspace{0.1cm}\\
		\sum_{i=1}^{k_1}i^2\lambda_it+\sum_{i=1}^{k_1}i\lambda_i\left(\sum_{i=1}^{k_1}i^2\lambda_i{t^2}+\zeta\sum_{i=1}^{k_1}i\lambda_i\frac{t^3}{3}\right),\ \eta=0,
	\end{cases}$
where $\eta=\sum_{i=1}^{k_1}i\lambda_1-\sum_{j=1}^{k_2}j\mu_j$, $\zeta=\sum_{i=1}^{k_1}i^2\lambda_i+\sum_{j=1}^{k_2}j^2\mu_j$ and $\xi=\sum_{i=1}^{k_1}i^2\lambda_i\sum_{j=1}^{k_2}j\mu_j+\sum_{i=1}^{k_1}i\lambda_i\sum_{j=1}^{k_2}j^2\mu_j$.
\end{proposition}
\begin{proof}
On comparing the coefficients on the both sides of the (\ref{comp}), we get
\begin{align*}
	\frac{\mathrm{d}}{\mathrm{d}t}\mathbb{E}({N}(t))&=\begin{cases}
		\eta\mathbb{E}({N}(t)),\ \eta\ne0,\vspace{0.1cm}\\
		0,\ \eta=0.
	\end{cases}\\
	\frac{\mathrm{d}}{\mathrm{d}t}\mathrm{Var}({N}(t))&=\begin{cases}
		2\eta\mathrm{Var}({N}(t))+\zeta\mathbb{E}({N}(t)),\ \eta\ne0,\vspace{0.1cm}\\
		\zeta,\ \eta=0.
	\end{cases}\\
	\frac{\mathrm{d}}{\mathrm{d}t}\mathrm{Cov}(B(t),{N}(t))&=\begin{cases}
		\sum_{i=1}^{k_1}i^2\lambda_i\mathbb{E}({N}(t))+\sum_{i=1}^{k_1}i\lambda_i\mathrm{Var}({N}(t))+\eta \mathrm{Cov}(B(t),{N}(t)),\ \eta\ne0,\vspace{0.1cm}\\
		\sum_{i=1}^{k_1}i^2\lambda_i+\sum_{i=1}^{k_1}i\lambda_i\zeta t,\ \eta=0.
	\end{cases}\\
	\frac{\mathrm{d}}{\mathrm{d}t}\mathbb{E}(B(t))&=\begin{cases}
		\sum_{i=1}^{k_1}i\lambda_i\mathbb{E}({N}(t)),\ \eta\ne0,\vspace{0.1cm}\\
		\sum_{i=1}^{k_1}i\lambda_i,\ \eta=0,
	\end{cases}	
\end{align*}
and	 
\begin{equation*}
	\frac{\mathrm{d}}{\mathrm{d}t}\mathrm{Var}(B(t))=\sum_{i=1}^{k_1}i^2\lambda_i\mathbb{E}({N}(t))+2\sum_{i=1}^{k_1}i\lambda_i\mathrm{Cov}(B(t),{N}(t)).
\end{equation*}	
On solving the above differential equations with initial conditions  $\mathbb{E}({N}(0))=1$, $\mathrm{Var}({N}(0))=0$, $\mathrm{Cov}(B(0),{N}(0))=0$, $\mathbb{E}(B(0))=1$ and $\mathrm{Var}(B(0))=0$, respectively, we get the required results. 
\end{proof}

On using (ii), (iii) and (v) of Proposition~\ref{prop3.3}, we can obtain the correlation coefficient between $B(t)$ and $N(t)$ as follows:
\begin{equation*}
	R_B(t)=\mathrm{Corr}(B(t),N(t))=\frac{\mathrm{Cov}(B(t),N(t))}{\sqrt{\mathrm{Var}(B(t))\mathrm{Var}(N(t))}}.
\end{equation*}
\begin{remark}\label{3.1}
	The mean and variance of GLBDP obtained in Proposition~\ref{prop3.2} coincide with (\ref{meanvar}). From the mean of $N(t)$ and $B(t)$, we get the following limiting results: 
	\begin{equation*}
		\lim_{t\to\infty}\mathbb{E}({N}(t))=\begin{cases}
			0,\ \eta<0,\vspace{0.1cm}\\
			1,\ \eta=0,\vspace{0.1cm}\\
			\infty,\ \eta>0, 
		\end{cases}
	\end{equation*}
	and
	\begin{equation*}
		\lim_{t\to\infty}\mathbb{E}(B(t))=\begin{cases}
			\infty,\  \eta\ge0,\\
			-\frac{\sum_{j=1}^{k_2}j\mu_j}{\eta},\ \eta<0,
		\end{cases}
	\end{equation*}
	respectively. Thus, in the long run, if $\eta>0$ then population is expected to grow beyond any bound, for $\eta<0$ the population is expected to grow up to some finite number then it is expected go extinct, and for $\eta=0$ it is expected that infinitely many people will be born however all of them will die out except one.	
\end{remark}
In the case of $k_1=k_2=2$, some numerical examples of cumulative births are given in Figure~\ref{fig1}.
\begin{figure}[!h]
	\includegraphics[width=13cm]{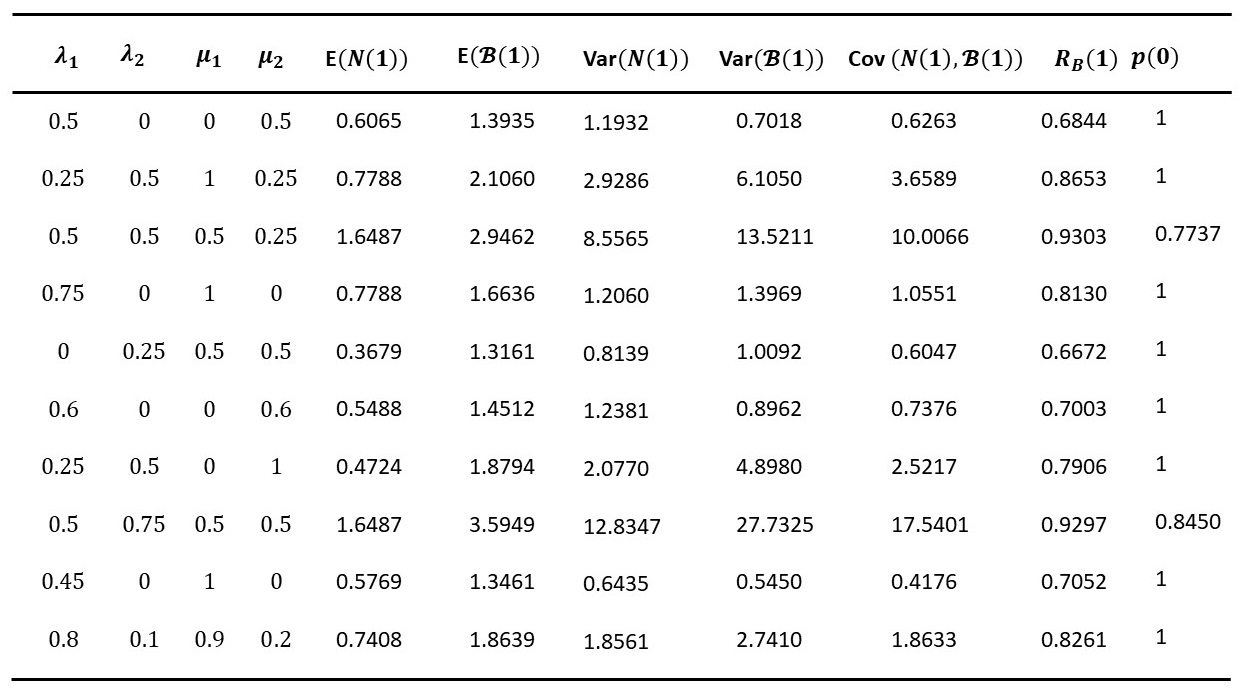}
	\caption[Figure 1]{\small Cumulative births in GLBDP for $k_1=k_2=2$ at $t=1$ .}\label{fig1}
\end{figure}

For different values of $\eta$, the variation of correlation coefficient $R_B(t)$ with respect to $t$ is shown in Figure~\ref{fig4}. It can be observed that if $\eta>0$ then correlation is high for the larger value of $\eta$ and for significantly large value of $t$, we have perfect positive correlation between $N(t)$ and $B(t)$. For $\eta<0$, $R_B(t)$ decreases with time and there is no linear correlation between $N(t)$ and $B(t)$ for the very large value of $t$.

\begin{figure}[htp]
	\includegraphics[width=16cm]{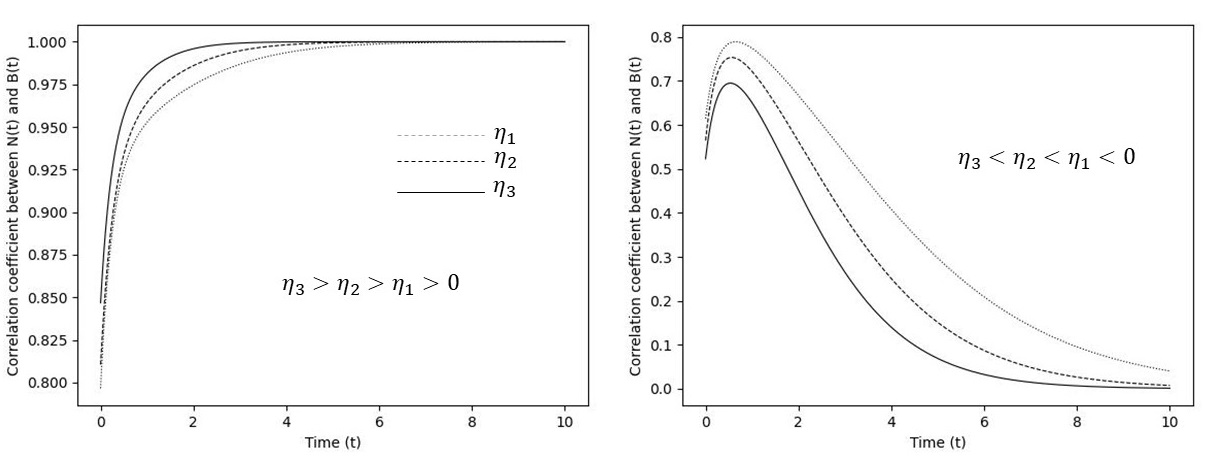}
	\caption[Figure 2]{Correlation between $N(t)$ and $B(t)$ for $k_1=k_2=2$}\label{fig4}
\end{figure}

Let $D(t)$ denotes the total number of deaths by time $t$.
\begin{proposition}\label{prop3.4}
	The joint pmf $q(d,n,t)=\mathrm{Pr}\{D(t)=d,{N}(t)=n\}$
is the solution of the following system of differential equations:
\begin{align}
	\frac{\mathrm{d}}{\mathrm{d}t}q(d,n,t)&=-n\left(\sum_{i=1}^{k_1}\lambda_i+\sum_{j=1}^{k_2}\mu_j\right)q(d,n,t)+\sum_{i=1}^{k_1}(n-i)\lambda_iq(d,n-i,t)\nonumber\\
	&\ \ +\sum_{j=1}^{k_2}(n+j)\mu_jq(d-j,n+j,t),\ d\ge0,\,n\ge0,\label{cumd}
\end{align}
with initial conditions $q(0,1,0)=1$ and $q(d,n,0)=0$ for all $d\ne0$ and $n\ne1$,
where $q(d,n,t)=0$ for all $d<0$ or $n<0$.
\end{proposition}
\begin{proof}
	In an infinitesimal interval of length $h$ such that $o(h)/h$ as $h\to0$, we have 
	\begin{multline*}
		\mathrm{Pr}\{D(t+h)=d+j,{N}(t+h)=n+k|D(t)=d,{N}(t)=n\}\\
		=\begin{cases}
			1-\bigg(\sum_{i=1}^{k_1}\lambda_i+\sum_{i=1}^{k_2}\mu_i\bigg)nh+o(h),\ j=k=0,\vspace{0.1cm}\\
			n\lambda_kh+o(h),\ j=0,\,k=1,2,\ldots,k_1,\vspace{0.1cm}\\
			n\mu_jh+o(h),\ k=-j,\,j=1,2,\ldots,k_2,\vspace{0.1cm}\\
			o(h),\ \mathrm{otherwise}.
		\end{cases}
	\end{multline*}
So,
\begin{align*}
	q(d,n,t+h)&=\left(1-\left(\sum_{i=1}^{k_1}\lambda_i+\sum_{j=1}^{k_2}\mu_j\right)nh\right)q(d,n,t)+\sum_{i=1}^{k_1}(n-i)h\lambda_ip(d,n-i,t)\\
	&\ \ +\sum_{j=1}^{k_2}(n+j)h\mu_jq(d-j,n+j,t)+o(h).
\end{align*}
After rearranging the terms and on letting $h\to0$, we get the required result.
\end{proof}

The proof of the next result is along the similar line to that of Proposition~\ref{prop3.2}, thus it is omitted.
\begin{proposition} 
	The pgf $\mathcal{G}(u,v,t)=\mathbb{E}\left(u^{D(t)}v^{{N}(t)}\right)$, $|u|\leq1$, $|v|\leq1$
solves 
\begin{equation*}
		\frac{\partial}{\partial t}\mathcal{G}(u,v,t)=\left(\sum_{i=1}^{k_1}\lambda_iv(v^i-1)+\sum_{j=1}^{k_2}\mu_jv((uv^{-1})^j-1)\right)\frac{\partial}{\partial v}\mathcal{G}(u,v,t)
\end{equation*}
with $\mathcal{G}(u,v,0)=v.$
\end{proposition}
Thus, the cgf $\mathcal{K}(\theta_u,\theta_v,t)=\ln\mathcal{G}(e^{\theta_u},e^{\theta_v},t)$, $\theta_u\leq0$, $\theta_v\leq0$ solves
	\begin{equation}\label{cumgd1}
			\frac{\partial}{\partial t}\mathcal{K}(\theta_u,\theta_v,t)=\left(\sum_{i=1}^{k_1}\lambda_i(e^{i\theta_v}-1)+\sum_{j=1}^{k_2}\mu_j(e^{j(\theta_u-\theta_v)}-1)\right)\frac{\partial}{\partial \theta_v}\mathcal{K}(\theta_u,\theta_v,t)
	\end{equation}
	with $	\mathcal{K}(\theta_u,\theta_v,0)={\theta_v}.$

On using (\ref{cgfequ}), we get
\begin{align}\label{cumgd2}
	\frac{\mathrm{d}}{\mathrm{d}t}\big(\theta_u\mathbb{E}(D(t))&+\theta_v\mathbb{E}({N}(t))+\frac{\theta_u^2}{2}\mathrm{Var}(D(t))+\frac{\theta_v^2}{2}\mathrm{Var}({N}(t))+\theta_u\theta_v\mathrm{Cov}(D(t),{N}(t))+\cdots\big)\nonumber\\
	&=\sum_{l=1}^{\infty}\bigg(\sum_{i=1}^{k_1}\lambda_i\frac{(i\theta_v)^l}{l!}+\sum_{j=1}^{k_2}\mu_j\frac{(j(\theta_u-\theta_v))^l}{l!}\bigg)\big(\mathbb{E}({N}(t))+\theta_v\mathrm{Var}({N}(t))\nonumber\\
	&\ \ +\theta_u\mathrm{Cov}(D(t),{N}(t))+\cdots\big)
\end{align}

On comparing the coefficients on the both side of the (\ref{cumgd2}), we get
\begin{align*}
	\frac{\mathrm{d}}{\mathrm{d}t}\mathbb{E}(D(t))&=\begin{cases}
		\sum_{j=1}^{k_2}j\mu_j\mathbb{E}(N(t)),\ \eta\ne0,\vspace{0.1cm}\\
		\sum_{j=1}^{k_2}j\mu_j,\ \eta=0,
	\end{cases}\\
	\frac{\mathrm{d}}{\mathrm{d}t}\mathrm{Cov}(D(t),{N}(t))&=\begin{cases}
		\sum_{j=1}^{k_2}\left(-j^2\mu_j\mathbb{E}({N}(t))+j\mu_j\mathrm{Var}({N}(t))\right)+\eta\mathrm{Cov}(D(t),{N}(t)),\ \eta\ne0,\vspace{0.1cm}\\
		\sum_{j=1}^{k_2}\left(-j^2\mu_j+j\mu_j\left(\sum_{i=1}^{k_1}i^2\lambda_i+\sum_{j=1}^{k_2}j^2\mu_j\right)t\right),\ \eta=0,
	\end{cases}
\end{align*}
and
\begin{equation*}\label{VD}
	\frac{\mathrm{d}}{\mathrm{d}t}\mathrm{Var}(D(t))=
	\sum_{j=1}^{k_2}j^2\mu_j\mathbb{E}(N(t))+2\sum_{j=1}^{k_2}j\mu_j\mathrm{Cov}(D(t),N(t)).
\end{equation*}
On solving these differential equations with initial conditions $\mathbb{E}(D(0))=0$, $\mathrm{Cov}(D(0),{N}(0))=0$ and $\mathrm{Var}(D(0))=0$, respectively, we get the following result.
\begin{proposition}\label{prop5}
	For $t\ge0$, we have

		\noindent (i) $\mathbb{E}(D(t))=\begin{cases}
			\frac{	\sum_{j=1}^{k_2}j\mu_j}{\eta}(e^{\eta t}-1),\ \eta\ne0,\vspace{0.1cm}\\
			\sum_{j=1}^{k_2}j\mu_jt,\ \eta=0,
		\end{cases}$\\
		\noindent (ii) $\mathrm{Cov}(D(t),{N}(t))=\begin{cases}
			\frac{\zeta\sum_{j=1}^{k_2}j\mu_j}{\eta^2}\left(e^{2\eta t}-e^{\eta t}\right)-\frac{\xi}{\eta}te^{\eta t},\ \eta\ne0,\vspace{0.1cm}\\
			-\sum_{j=1}^{k_2}j^2\mu_jt+\zeta\sum_{j=1}^{k_2}j\mu_j\frac{t^2}{2},\ \eta=0,
		\end{cases}$\\
		\noindent (iii) $\mathrm{Var}(D(t))=\begin{cases}
			\frac{\sum_{j=1}^{k_2}j^2\mu_j}{\eta}(e^{\eta t}-1)+2\sum_{j=1}^{k_2}j\mu_j\bigg(\frac{\zeta \sum_{j=1}^{k_2}j\mu_j}{2\eta^3}\left(e^{\eta t}-1\right)^2-\frac{\xi}{\eta^3}\left(\eta t e^{\eta t}-e^{\eta t}+1\right)\bigg),\ \eta\ne0,\vspace{0.1cm}\\
			\sum_{j=1}^{k_2}j^2\mu_jt+\sum_{j=1}^{k_2}j\mu_j\left(-\sum_{j=1}^{k_2}j^2\mu_j{t^2}+\zeta\sum_{j=1}^{k_2}j\mu_j\frac{t^3}{3}\right),\ \eta=0.
		\end{cases}$
\end{proposition}

Here, again using the results of Proposition~\ref{prop3.2} and Proposition~\ref{prop5}, we can calculate the correlation coefficient ${R_D}(t)=\mathrm{Corr}(D(t),N(t))$ between $D(t)$ and $N(t)$. For $k_1=k_2=2$, some numerical examples of cumulative deaths are given in Figure~\ref{fig2}.
\begin{figure}[htp]
	\includegraphics[width=13cm]{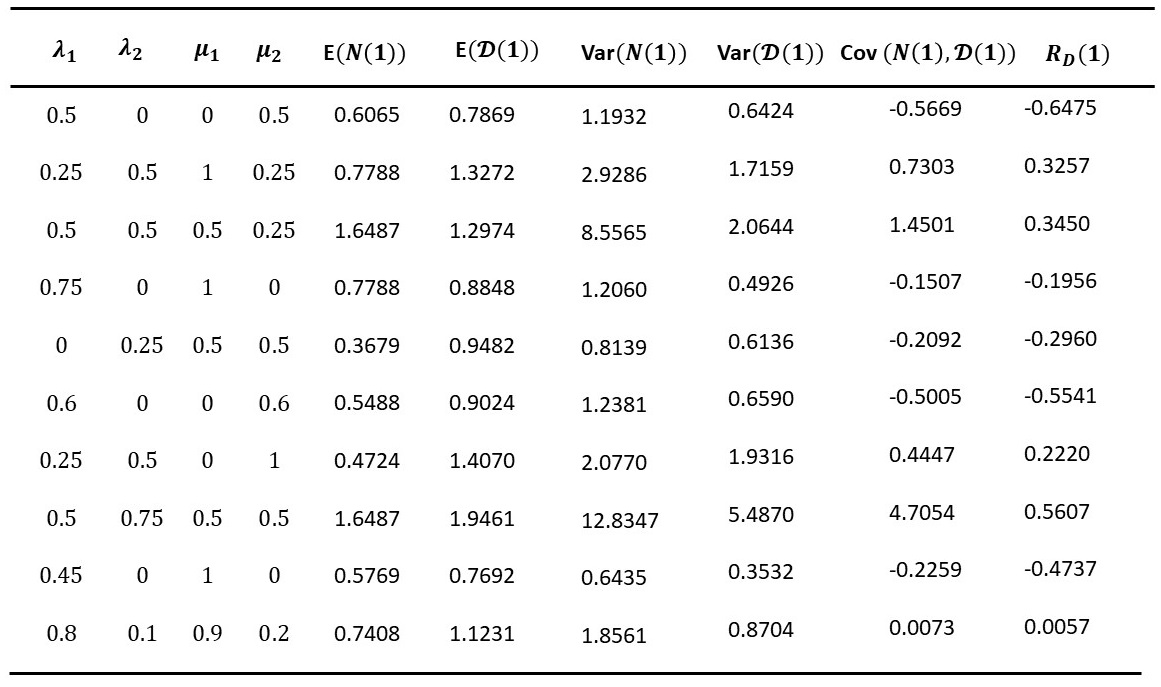}
	\caption[Figure2]{\small Cumulative deaths in GLBDP for $k_1=k_2=2$ at time $t=1$.}\label{fig2}
\end{figure}

From Figure~\ref{fig5}, we can observe that for $\eta>0$, as $t$ increases the correlation coefficient $R_D(t)$ changes from negative to positive and it is strong for large values of $t$. For $\eta<0$ there is no correlation between $N(t)$ and $D(t)$.

\begin{figure}[htp]
	\includegraphics[width=16cm]{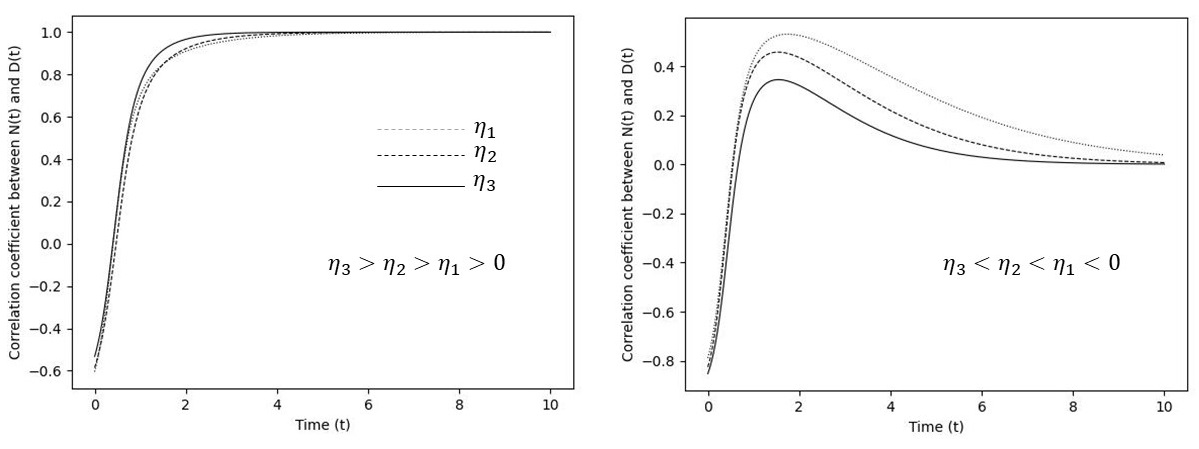}
	\caption[Figure 5]{Correlation between $N(t)$ and $D(t)$ for $k_1=k_2=2$}\label{fig5}
\end{figure}

\begin{proposition}\label{prop3.7}
	The joint pmf $	p(d,b,n,t)=\mathrm{Pr}\{D(t)=d,B(t)=b,{N}(t)=n\}$ is the solution of the following system of differential equations:
	\begin{align}
		\frac{\mathrm{d}}{\mathrm{d}t}p(d,b,n,t)&=-\bigg(\sum_{i=1}^{k_1}\lambda_i+\sum_{j=1}^{k_2}\mu_j\bigg)np(d,b,n,t)+\sum_{i=1}^{k_1}(n-i)\lambda_ip(d,b-i,n-i,t)\nonumber\\
		&\ \ +\sum_{j=1}^{k_2}(n+j)\mu_jp(d-j,b,n+j,t),\ d\ge0,\,b\ge1,\,n\ge0,\label{jointdbn}
	\end{align}
	with initial conditions $p(0,1,1,0)=1$ and $p(d,b,n,0)=0$ for all $d\ne0$, $b\ne1$ and $n\ne1$,
	where $p(d,b,n,t)=0$ for $d<0$ or $b<0$ or $n<0$.
\end{proposition}
\begin{proof}
	Let  $h$ be the length of an infinitesimal interval such that $o(h)/h\to0$ as $h\to0$. Then,  
	\begin{multline*}
		\mathrm{Pr}\{D(t+h)=d+j,B(t+h)=b+k,{N}(t+h)=n+m|D(t)=d,B(t)=b,{N}(t)=n\}\\
		=\begin{cases}
			1-\left(\sum_{i=1}^{k_1}\lambda_i+\sum_{i=1}^{k_2}\mu_i\right)nh+o(h),\ j=k=m=0,\vspace{0.1cm}\\
			n\lambda_kh+o(h),\ j=0,\,k=m=1,2,\ldots,k_1,\vspace{0.1cm}\\
			n\mu_jh+o(h),\ k=0,\,m=-j,\,j=1,2,\ldots,k_2,\vspace{0.1cm}\\
			o(h),\ \mathrm{otherwise}.
		\end{cases}
	\end{multline*}
	So,
	\begin{align*}
		p(d,b,n,t+h)&=\bigg(1-\bigg(\sum_{i=1}^{k_1}\lambda_i+\sum_{j=1}^{k_2}\mu_j\bigg)nh\bigg)p(d,b,n,t)+\sum_{i=1}^{k_1}(n-i)\lambda_ihp(d,b-i,n-i,t)\\
		&\ \ +\sum_{j=1}^{k_2}(n+j)\mu_jhp(d-j,b,n+j,t)+o(h).
	\end{align*}
	After rearranging the terms and on letting $h\to0$, we get the required result.
\end{proof}
	The pgf $	U(u,v,w,t)=\mathbb{E}(u^{D(t)}v^{B(t)}w^{{N}(t)})$, $|u|\leq1$, $|v|\leq1$, $|w|\leq1$ of $(D(t),B(t),N(t))$
	satisfies the following differential equation:
	\begin{equation*}
		\frac{\partial}{\partial t}U(u,v,w,t)=\bigg(\sum_{i=1}^{k_1}\lambda_iw((vw)^i-1)+\sum_{j=1}^{k_2}\mu_jw((uw^{-1})^j-1)\bigg)\frac{\partial}{\partial w}U(u,v,w,t)
	\end{equation*}
	with initial condition $U(u,v,w,0)=vw.$
	
	As $N(t)=B(t)-D(t)$, $t\ge0$, we have 
	$
		\mathrm{Var}(N(t))=\mathrm{Var}(D(t))+\mathrm{Var}(B(t))-2\mathrm{Cov}(D(t),B(t)).
	$
	So, the covariance of $B(t)$ and $D(t)$, $t\ge0$ is 
	\begin{align*}
		\mathrm{Cov}(D(t),B(t))&=\frac{1}{2}(\mathrm{Var}(D(t))+\mathrm{Var}(B(t))-\mathrm{Var}(N(t)))\\
		&=\frac{1}{2}\Bigg(\frac{\zeta}{\eta^3}\Bigg(\bigg(\sum_{i=1}^{k_1}i\lambda_i\bigg)^2+\bigg(\sum_{j=1}^{k_2}j\mu_j\bigg)^2\Bigg)\left(e^{\eta t}-1\right)^2-\frac{\zeta}{\eta}\left(e^{\eta t}-1\right)^2\\
		&\ \ -\frac{2\xi}{\eta^3}\Bigg(\sum_{i=1}^{k_1}i\lambda_i+\sum_{j=1}^{k_2}j\mu_j\Bigg)\left(\eta t e^{\eta t}-e^{\eta t}+1\right)\Bigg)\\
		&=\frac{\zeta\sum_{i=1}^{k_1}\sum_{j=1}^{k_2}\lambda_i\mu_j}{\eta^3}\left(e^{\eta t}-1\right)^2-\frac{\xi}{\eta^3}\Bigg(\sum_{i=1}^{k_1}i\lambda_i+\sum_{j=1}^{k_2}j\mu_j\Bigg)\left(\eta t e^{\eta t}-e^{\eta t}+1\right),\ \eta\ne0,
	\end{align*}
	which can further be used to calculate the correlation coefficient $R(t)=\mathrm{Corr}(D(t),B(t))$. 
	
	Next, we do the similar study for $\{N^*(t)\}_{t\ge0}$, that is, the GBDP with constant birth and death rates.
\subsection{GBDP with constant birth and death rates}
Let $B^*(t)$ and $D^*(t)$ be the total number of births and deaths up to time $t$ in the process $\{N^*(t)\}_{t\ge0}$, respectively. Similar to the result of Proposition~\ref{prop3.7}, the joint pmf $p^*(d,b,n,t)=\mathrm{Pr}\{D^*(t)=d,B^*(t)=b,N^*(t)=n\}$ solves the following  system of differential equations:
	\begin{align*}
		\frac{\mathrm{d}}{\mathrm{d}t}p^*(d,b,n,t)&=-\bigg(\sum_{i=1}^{k_1}\lambda_i+\sum_{j=1}^{k_2}\mu_j\bigg)p^*(d,b,n,t)+\sum_{i=1}^{k_1}\lambda_ip^*(d,b-i,n-i,t)\\
		&\ \ +\sum_{j=1}^{k_2}\mu_jp^*(d-j,b,n+j,t),\ d\ge0,\,b\ge1,\,n\ge0
	\end{align*}
	with initial conditions $p^*(0,1,1,0)=1$ and $p^*(d,b,n,t)=0$ for all $d\ne0$, $b\ne1$ and $n\ne1$.
	
	So, the pgf $U^*(u,v,w,t)=\mathbb{E}(u^{D^*(t)}v^{B^*(t)}w^{N^*(t)})$, $|u|\leq1$, $|v|\leq1$, $|w|\leq1$ is the solution of the following differential equation:
	\begin{equation*}
		\frac{\mathrm{d}}{\mathrm{d}t}U^*(u,v,w,t)=\left(\sum_{i=1}^{k_1}\lambda_i\left((vw)^i-1\right)+\sum_{j=1}^{k_2}\mu_j\left((uw^{-1})^j-1\right)\right)U^*(u,v,w,t)
	\end{equation*}
	with $U^*(u,v,w,0)=vw$. It is given by
	\begin{equation}\label{pgfdbn}
		U^*(u,v,w,t)=vwe^{-\Lambda t}\exp{\left(\sum_{i=1}^{k_1}\lambda_i(vw)^i+\sum_{j=1}^{k_2}\mu_j(uw^{-1})^j\right)t}.
	\end{equation}
	
\begin{theorem}
	The joint pmf $p^*(d,b,n,t)=\mathrm{Pr}\{D^*(t)=d,B^*(t)=b,N^*(t)=n\}$ is given by
	\begin{equation}\label{pmfdbn}
		p^*(d,b,n,t)=\sum_{\Omega^*(k_1,k_2,d,b,n)}\prod_{i=1}^{k_1}\prod_{j=1}^{k_2}\frac{\lambda_i^{x_i}\mu_j^{y_j}t^{x_i+y_j}}{x_i!y_j!}e^{-\Lambda t},\ d\ge0,\,b\ge1,\,n\ge0,
	\end{equation}
	where
	$
		\Omega^*(k_1,k_2,d,b,n)=\bigg\{(x_1,\ldots,x_{k_1},y_1,\ldots,y_{k_2})\in \mathbb{N}_0^{k_1+k_2}:\sum_{j=1}^{k_2}jy_j=d,\,\sum_{i=1}^{k_1}ix_i=b-1,\,b-d=n\bigg\}.
	$
\end{theorem}
\begin{proof}
	From (\ref{pgfdbn}), we have
	\begin{equation*}
		U^*(u,v,w,t)=vwe^{-\Lambda t}\sum_{l=0}^{\infty}\bigg(\sum_{i=1}^{k_1}\lambda_i(vw)^i+\sum_{j=1}^{k_2}\mu_j(uw^{-1})^j\bigg)^l\frac{t^l}{l!}.
	\end{equation*}
Equivalently,
	\begin{equation*}
			U^*(u,v,w,t)=e^{-\Lambda t}\sum_{l=0}^{\infty}\sum_{S^*(k_1,k_2,l)}\prod_{i=1}^{k_1}\prod_{j=1}^{k_2}\frac{\lambda_i^{x_i}\mu_j^{y_j}t^{x_i+y_j}}{x_i!y_j!}u^{jy_j}v^{ix_i+1}w^{ix_i-jy_j+1},
	\end{equation*}
where 
$
		S^*(k_1,k_2,l)=\big\{(x_1,\ldots,x_{k_1},y_1,\ldots,y_{k_2}):x_i,y_j\in\{0,1,\ldots,l\},\,\sum_{i=1}^{k_1}x_i+\sum_{j=1}^{k_2}y_j=l\big\}.
$ 	
So, the coefficients of $u^dv^bw^n$ in $U^*(u,v,w,t)$ is the required result.
\end{proof}
\begin{remark}
	From (\ref{pmfdbn}), we get the following joint pmfs
	\begin{align*}
		p^*(b,n,t)&=\mathrm{Pr}\{B^*(t)=b,N^*(t)=n\}=\sum_{\Omega^*(k_1,k_2,b,n)}\prod_{i=1}^{k_1}\prod_{j=1}^{k_2}\frac{\lambda_i^{x_i}\mu_j^{y_j}t^{x_i+y_j}}{x_i!y_j!}e^{-\Lambda t},\ b\ge1,\,n\ge0,\\
		q^*(d,n,t)&=\mathrm{Pr}\{D^*(t)=b,N^*(t)=n\}=\sum_{\Omega^{**}(k_1,k_2,d,n)}\prod_{i=1}^{k_1}\prod_{j=1}^{k_2}\frac{\lambda_i^{x_i}\mu_j^{y_j}t^{x_i+y_j}}{x_i!y_j!}e^{-\Lambda t},\ d\ge0,\,n\ge0,
	\end{align*}
	and 
	\begin{equation*}
		\tilde{p}^*(d,b,t)=\mathrm{Pr}\{D^*(t)=d,B^*(t)=b\}=\sum_{\Omega^{***}(k_1,k_2,d,b)}\prod_{i=1}^{k_1}\prod_{j=1}^{k_2}\frac{\lambda_i^{x_i}\mu_j^{y_j}t^{x_i+y_j}}{x_i!y_j!}e^{-\Lambda t},\ d\ge0,\,b\ge1,
	\end{equation*}
	where 
	\begin{align*}			\Omega^*(k_1,k_2,b,n)&=\bigg\{(x_1,\ldots,x_{k_1},y_1,\ldots,y_{k_2})\in \mathbb{N}_0^{k_1+k_2}:\sum_{i=1}^{k_1}ix_i=b-1,\,\sum_{i=1}^{k_1}ix_i-\sum_{j=1}^{k_2}jy_j=n-1\bigg\},\\
		\Omega^{**}(k_1,k_2,d,n)&=\bigg\{(x_1,\ldots,x_{k_1},y_1,\ldots,y_{k_2})\in \mathbb{N}_0^{k_1+k_2}:\sum_{j=1}^{k_2}jy_j=d,\,\sum_{i=1}^{k_1}ix_i-\sum_{j=1}^{k_2}jy_j=n-1\bigg\},
	\end{align*}
	and
	\begin{equation*}
		\Omega^{***}(k_1,k_2,d,b)=\bigg\{(x_1,\ldots,x_{k_1},y_1,\ldots,y_{k_2})\in \mathbb{N}_0^{k_1+k_2}:\sum_{j=1}^{k_2}jy_j=d,\,\sum_{i=1}^{k_1}ix_i=b-1\bigg\}.
	\end{equation*}
\end{remark}
\begin{remark} From (\ref{pmfdbn}), we obtain the following marginal pmfs:
	\begin{equation*}
		\tilde{p}^*(d,t)=\mathrm{Pr}\{D^*(t)=d\}=\sum_{\Omega^{*}(k_2,d)}\prod_{j=1}^{k_2}\frac{(\mu_jt)^{y_j}}{y_j!}e^{-\sum_{j=1}^{k_2}\mu_jt},\ d\ge0,
	\end{equation*}
	and 
	\begin{equation*}
		{p}^*(b,t)=\mathrm{Pr}\{B^*(t)=b\}=\sum_{\Omega^{**}(k_1,b)}\prod_{i=1}^{k_1}\frac{(\lambda_it)^{x_i}}{x_i!}e^{-\sum_{i=1}^{k_1}\lambda_it},\ b\ge1,
	\end{equation*}
	where 
	$
		\Omega^{*}(k_2,d)=\bigg\{(y_1,\ldots,y_{k_2})\in \mathbb{N}_0^{k_2}:\sum_{j=1}^{k_2}jy_j=d\bigg\}$ and $	\Omega^{**}(k_1,b)=\bigg\{(x_1,\ldots,x_{k_1})\in \mathbb{N}_0^{k_1}:\sum_{i=1}^{k_1}ix_i=b-1\bigg\}.
	$
	\end{remark}
	\begin{remark}
		Note that the joint pmf of $(D^*(t),B^*(t))$ is the product of their marginal pmfs. Hence,  ${\{D^*(t)\}}_{t\ge0}$ and ${\{B^*(t)\}}_{t\ge0}$ are two independent random processes. 
	\end{remark}
	\begin{proposition}
		For $t\ge0$, we have\\
		\noindent (i) $\mathbb{E}(B^*(t))=1+\sum_{i=1}^{k_1}i\lambda_it$,\\
		\noindent (ii) $\mathbb{E}(D^*(t))=\sum_{j=1}^{k_2}j\mu_jt$,\\
		\noindent (iii) $\mathrm{Var}(B^*(t))=\sum_{i=1}^{k_1}i^2\lambda_it$,\\
		\noindent (iv) $\mathrm{Var}(D^*(t))=\sum_{j=1}^{k_2}j^2\mu_jt$.
	\end{proposition}
	\begin{proof}
		On taking derivatives of (\ref{pgfdbn}) with respect to $u$ and $v$, we get
		\begin{align*}
			\frac{\partial}{\partial u}U^*(u,v,w,t)&=vwe^{-\Lambda t}e^{\left(\sum_{i=1}^{k_1}\lambda_i(vw)^i+\sum_{j=1}^{k_2}\mu_j(uw^{-1})^j\right)t}\sum_{j=1}^{k_2}j\mu_ju^{j-1}w^{-j}t,\\
			\frac{\partial^2}{\partial u^2}U^*(u,v,w,t)&=vwe^{-\Lambda t}e^{\left(\sum_{i=1}^{k_1}\lambda_i(vw)^i+\sum_{j=1}^{k_2}\mu_j(uw^{-1})^j\right)t}\Bigg(\Bigg(\sum_{j=1}^{k_2}j\mu_ju^{j-1}w^{-j}\Bigg)^2t^2\\
			&\ \ +\sum_{j=1}^{k_2}j(j-1)\mu_ju^{j-2}w^{-j}t\Bigg),\\
			\frac{\partial}{\partial v}U^*(u,v,w,t)&=we^{-\Lambda t}e^{\left(\sum_{i=1}^{k_1}\lambda_i(vw)^i+\sum_{j=1}^{k_2}\mu_j(uw^{-1})^j\right)t}\left(1+\sum_{i=1}^{k_1}i\lambda_i(vw)^it\right),
		\end{align*}
		and
		\begin{align*}
			\frac{\partial}{\partial v^2}U^*(u,v,w,t)&=we^{-\Lambda t}e^{\left(\sum_{i=1}^{k_1}\lambda_i(vw)^i+\sum_{j=1}^{k_2}\mu_j(uw^{-1})^j\right)t}\Bigg(\sum_{i=1}^{k_1}i^2\lambda_iv^{i-1}w^it\\
			&\ \ +\Bigg(1+\sum_{i=1}^{k_1}i\lambda_i(vw)^it\Bigg)\sum_{i=1}^{k_1}i\lambda_iv^{i-1}w^it\Bigg).
		\end{align*}
		So,
		\begin{align*}
			\mathbb{E}(D^*(t))&=\frac{\partial}{\partial u}U^*(u,v,w,t)\bigg|_{u=1,v=1,w=1}=\sum_{j=1}^{k_2}j\mu_jt,\\
			\mathbb{E}(B^*(t))&=\frac{\partial}{\partial v}U^*(u,v,w,t)\bigg|_{u=1,v=1,w=1}=1+\sum_{i=1}^{k_1}i\lambda_it,\\
			\mathrm{Var}(D^*(t))&=\frac{\partial^2}{\partial u^2}U^*(u,v,w,t)\bigg|_{u=1,v=1,w=1}+\mathbb{E}(D^*(t))-(\mathbb{E}(D^*(t)))^2=\sum_{j=1}^{k_2}j^2\mu_jt,
		\end{align*}
		and 
		\begin{equation*}
		\mathrm{Var}(B^*(t))=\frac{\partial^2}{\partial v^2}U^*(u,v,w,t)\bigg|_{u=1,v=1,w=1}+\mathbb{E}(N^*(t))-(\mathbb{E}(B^*(t)))^2=\sum_{i=1}^{k_1}i^2\lambda_it.
		\end{equation*}
		This completes the proof.
	\end{proof}
	\begin{proposition}
		The variance-covariance matrix $\Sigma$ of $(D^*(t),B^*(t),N^*(t))$ is given by
		\begin{equation*}
			\Sigma=\begin{pmatrix}
				\sum_{j=1}^{k_2}j^2\mu_jt&0&-\sum_{j=1}^{k_2}j^2\mu_jt\vspace{0.2cm}\\
				0&\sum_{i=1}^{k_1}i^2\lambda_it&\sum_{i=1}^{k_1}i^2\lambda_it\vspace{0.2cm}\\
				-\sum_{j=1}^{k_2}j^2\mu_jt&\sum_{i=1}^{k_1}i^2\lambda_it&\left(\sum_{i=1}^{k_1}i^2\lambda_i+\sum_{j=1}^{k_2}j^2\mu_j\right)t
			\end{pmatrix}
		\end{equation*}
		for all $t\ge0$.
	\end{proposition}
	\begin{proof}
			On taking derivatives of (\ref{pgfdbn}) with respect to $u,\,v$ and $w$, we get
			\begin{align*}
				\frac{\partial^2}{\partial u\partial v}U^*(u,v,w,t)&=we^{-\Lambda t}e^{\left(\sum_{i=1}^{k_1}\lambda_i(vw)^i+\sum_{j=1}^{k_2}\mu_j(uw^{-1})^j\right)t}\left(1+\sum_{i=1}^{k_1}i\lambda_i(vw)^it\right)\sum_{j=1}^{k_2}j\mu_ju^{j-1}w^{-j}t,\\
				\frac{\partial^2}{\partial u\partial w}U^*(u,v,w,t)&=ve^{-\Lambda t}e^{\left(\sum_{i=1}^{k_1}\lambda_i(vw)^i+\sum_{j=1}^{k_2}\mu_j(uw^{-1})^j\right)t}\bigg(\sum_{j=1}^{k_2}j\mu_ju^{j-1}w^{-j}t\bigg(\sum_{i=1}^{k_1}i\lambda_i(vw)^i\\
				&\ \ -\sum_{j=1}^{k_2}j\mu_j(uw^{-1})^j\bigg)t-\sum_{j=1}^{k_2}j(j-1)\mu_ju^{j-1}w^{-j}t\bigg),\\
				\frac{\partial^2}{\partial v\partial w}U^*(u,v,w,t)&=e^{-\Lambda t}e^{\left(\sum_{i=1}^{k_1}\lambda_i(vw)^i+\sum_{j=1}^{k_2}\mu_j(uw^{-1})^j\right)t}\Bigg(\Bigg(1+\sum_{i=1}^{k_1}i(i+1)\lambda_i(vw)^it\Bigg)\\
				&\ \ +\Bigg(1+\sum_{i=1}^{k_1}i\lambda_i(vw)^it\Bigg)\Bigg(\sum_{i=1}^{k_1}i\lambda_i(vw)^i-\sum_{j=1}^{k_2}j\mu_j(uw^{-1})^j\Bigg)t\Bigg).
			\end{align*}
			So,
			\begin{align*}
				\mathrm{Cov}(D^*(t),B^*(t))&=\frac{\partial^2}{\partial u\partial v}U^*(u,v,w,t)\bigg|_{u=1,v=1,w=1}-\mathbb{E}(D^*(t))\mathbb{E}(B^*(t))\\
				&=\left(1+\sum_{i=1}^{k_1}i\lambda_it\right)\sum_{j=1}^{k_2}j\mu_jt-\left(1+\sum_{i=1}^{k_1}i\lambda_it\right)\sum_{j=1}^{k_2}j\mu_jt=0,
			\end{align*}
			\begin{align*}
				\mathrm{Cov}(D^*(t),N^*(t))&=\frac{\partial^2}{\partial u\partial w}U^*(u,v,w,t)\bigg|_{u=1,v=1,w=1}-\mathbb{E}(D^*(t))\mathbb{E}(N^*(t))\\
				&=\sum_{j=1}^{k_2}j\mu_jt\left(\sum_{i=1}^{k_1}i\lambda_i-\sum_{j=1}^{k_2}j\mu_j\right)t-\sum_{j=1}^{k_2}j(j-1)\mu_jt\\
				&\ \  -\sum_{j=1}^{k_2}j\mu_jt\left(1+\left(\sum_{i=1}^{k_1}i\lambda_i-\sum_{j=1}^{k_2}j\mu_j\right)t\right)\\
				&=-\sum_{j=1}^{k_2}j^2\mu_jt,
			\end{align*}
			and
			\begin{align*}
				\mathrm{Cov}(B^*(t),N^*(t))&=\frac{\partial^2}{\partial v\partial w}U^*(u,v,w,t)\bigg|_{u=1,v=1,w=1}-\mathbb{E}(B^*(t))\mathbb{E}(N^*(t))\\
				&=1+\sum_{i=1}^{k_1}i(i+1)\lambda_it+\left(1+\sum_{i=1}^{k_1}i\lambda_it\right)\left(\sum_{i=1}^{k_1}i\lambda_i-\sum_{j=1}^{k_2}j\mu_j\right)t\\
				&\ \ -\left(1+\sum_{i=1}^{k_1}i\lambda_it\right)\left(1+\left(\sum_{i=1}^{k_1}i\lambda_i-\sum_{j=1}^{k_2}j\mu_j\right)t\right)\\
				&=\sum_{i=1}^{k_1}i^2\lambda_it.
			\end{align*}
			This completes the proof.
	\end{proof}

\section{Stochastic integral of GBDP }\label{sec5}
Puri (1966) and Puri (1968) studied the joint distribution of linear birth-death process and its integral. McNeil (1970) obtained the distribution of the integral functional of BDP. Moreover, Gani and McNeil (1971) studied the joint distributions of random variables and their integrals for certain birth-death
and diffusion processes.
Doubleday (1973) has done similar study for the linear birth-death process with multiple births.

 In GBDP, let $\mathcal{N}(0)=k$, that is, there are $k$ individuals present in a population at time $t=0$. The hitting time of $\{\mathcal{N}(t)\}_{t\ge}$ to state $0$ given that $\mathcal{N}(0)=k$ is defined as
  \begin{equation}\label{St}
	Z_k=\mathrm{inf}\{t\ge0:\mathcal{N}(t)=0\}.
\end{equation}
Let 
$
	p_k(t)=\mathrm{Pr}\{\mathcal{N}(t)=0|\mathcal{N}(0)=k\}.
$
Suppose that there is no possibility of immigration, that is, $\lambda_{(0)_i}=0$ for all $i=1,2,\ldots,k_1$. Then,
$
	\mathrm{Pr}\{Z_k\leq s\}=p_k(s),\ s\ge0.
$

Let us consider an infinitesimal time interval of length $h$. Then, we have
\begin{align*}
	p_k(t+h)&=\mathrm{Pr}\{\mathcal{N}(h)=k|\mathcal{N}(0)=k\}p_k(t)+\sum_{i=1}^{k_1}\mathrm{Pr}\{\mathcal{N}(h)=k+i|\mathcal{N}(0)=k\}p_{k+i}(t)\\
	&\ \ +\sum_{j=1}^{k_2}\mathrm{Pr}\{\mathcal{N}(h)=k-j|\mathcal{N}(0)=k\}p_{k-j}(t).
\end{align*}
On using (\ref{gbdp}), we get
\begin{align*}
	p_k(t+h)&=\bigg(1-\sum_{i=1}^{k_1}\lambda_{(k)_i }h+\sum_{j=1}^{k_2}\mu_{(k)_j}h\bigg)p_k(t)+\sum_{i=1}^{k_1}\lambda_{(k)_i}hp_{k+i}(t)+\sum_{j=1}^{k_2}\mu_{(k)_j}hp_{k-j}(t)+o(h),
\end{align*}
which on rearranging the terms and letting $h\to0$ reduces to
\begin{equation}\label{stopping}
	\frac{\mathrm{d}}{\mathrm{d}t}p_k(t)=-\bigg(\sum_{i=1}^{k_1}\lambda_{(k)_i}+\sum_{j=1}^{k_2}\mu_{(k)_j}\bigg)p_k(t)+\sum_{i=1}^{k_1}\lambda_{(k)_i}p_{k+i}(t)+\sum_{j=1}^{k_2}\mu_{(k)_j}p_{k-j}(t).
\end{equation}

Let 
\begin{equation}\label{g}
	\mathcal{W}_k=\int_{0}^{Z_k}g(\mathcal{N}(t))\,\mathrm{d}t,
\end{equation}
where $g(\cdot)$ is a non-negative real valued function on the state-space of GBDP. 
Observe that if $g(u)=1$ then $\mathcal{W}_k=Z_k$ and if $g(u)=u$ then $\mathcal{W}_k$ is the area under $\mathcal{N}(t)$ up to time of first extinction of the process given that it start from state $k$.

Now, let us define
$
	\mathcal{W}_k(\theta)=\mathbb{E}(e^{-\theta \mathcal{W}_k}),\ \theta\ge0.
$ Recall that $T_s$ denotes the total time spent by the process $\{\mathcal{N}(t)\}_{t\ge0}$ in state $s$ before the first transition takes place given that it start from state $s$, and let $B_i$ denotes the event that the first transition is $i$ many births and $D_j$ be the event that the first transition is $j$ many deaths given that we start from $k$. 
Then, we have
\begin{align*}
	\mathcal{W}_k(\theta)&=\sum_{i=1}^{k_1}\mathbb{E}(e^{-\theta \mathcal{W}_k}|B_i)\mathrm{Pr}(B_i)+\sum_{j=1}^{k_2}\mathbb{E}(e^{-\theta \mathcal{W}_k}|D_j)\mathrm{Pr}(D_j)\\
	&=\sum_{i=1}^{k_1}\mathbb{E}(e^{-\theta(\mathcal{W}_{k+i}+T_kg(k))})\mathrm{Pr}(B_i)+\sum_{j=1}^{k_2}\mathbb{E}(e^{-\theta(\mathcal{W}_{k-j}+T_kg(k))})\mathrm{Pr}(D_j)\\
	&=\sum_{i=1}^{k_1}\mathcal{W}_{k+i}(\theta)\mathbb{E}(e^{-\theta T_kg(k)})\mathrm{Pr}(B_i)+\sum_{j=1}^{k_2}\mathcal{W}_{k-j}(\theta)\mathbb{E}(e^{-\theta T_kg(k)})\mathrm{Pr}(D_j),
\end{align*}
where we used strong Markov property to get the last two steps. From the Markovian property of GBDP, we have
$
	\mathrm{Pr}(B_i)={\lambda_{(k)_i}}/{\rho}$ and $  \mathrm{Pr}(D_j)={\mu_{(k)_j}}/{\rho}, 
$ 
where $\rho=\sum_{i=1}^{k_1}\lambda_{(k)_i}+\sum_{j=1}^{k_2}\mu_{(k)_j}$. On using Remark~\ref{exp}, we get
\begin{align*}
	\mathcal{W}_k(\theta)&=\bigg(\sum_{i=1}^{k_1}\mathcal{W}_{k+i}(\theta)\lambda_{(k)_i}+\sum_{j=1}^{k_2}\mathcal{W}_{k-j}(\theta)\mu_{(k)_j}\bigg)\int_{0}^{\infty}e^{-\theta sg(k)}e^{-\rho s}\,\mathrm{d}s\\
&=\bigg(\sum_{i=1}^{k_1}\mathcal{W}_{k+i}(\theta)\lambda_{(k)_i}+\sum_{j=1}^{k_2}\mathcal{W}_{k-j}(\theta)\mu_{(k)_j}\bigg)\frac{1}{\theta g(k)+\rho}.
\end{align*}
Hence,
\begin{equation}\label{laplac}
	\theta\mathcal{W}_k(\theta)=\sum_{i=1}^{k_1}\lambda_{ik}^*\mathcal{W}_{k+i}(\theta)+\sum_{j=1}^{k_2}\mu_{jk}^*\mathcal{W}_{k-j}(\theta)-\bigg(\sum_{i=1}^{k_1}\lambda_{ik}^*+\sum_{j=1}^{k_2}\mu_{jk}^*\bigg)\mathcal{W}_k(\theta)
\end{equation}
with 
$
	\mathcal{W}_0=\int_{0}^{Z_0}g(\mathcal{N}(t))\,\mathrm{d}t=0$ and $ \mathcal{W}_0(\theta)=1.	
$
Here,
$
	\lambda_{ik}^*={\lambda_{(k)_i}}/{g(k)}$ and $\mu_{jk}^*={\mu_{(k)_j}}/{g(k)}.
$

As McNeil (1970) pointed out that Eq. (\ref{laplac}) correspond to finding the Laplace transform of (\ref{stopping}) with parameters $\lambda_{ik}^*$ and $\mu_{jk}^*$ in place of $\lambda_{(k)_i}$ and $\mu_{(k)_j}$, respectively, that is, on taking the Laplace transform of (\ref{stopping}), we get
\begin{equation}
	\theta \mathcal{L}_k(\theta)=\sum_{i=1}^{k_1}\lambda_{(k)_i}\mathcal{L}_{k+i}(\theta)+\sum_{j=1}^{k_2}\mu_{(k)_j}\mathcal{L}_{k-j}(\theta)-\bigg(\sum_{i=1}^{k_1}\lambda_{(k)_i}+\sum_{j=1}^{k_2}\mu_{(k)_j}\bigg)\mathcal{L}_k(\theta),
\end{equation}
where 
\begin{equation*}
	\mathcal{L}_k(\theta)=\int_{0}^{\infty}e^{-\theta t}\bigg(\frac{\mathrm{d}}{\mathrm{d}t}p_k(t)\bigg)\,\mathrm{d}t=\mathbb{E}(e^{-\theta Z_k})
\end{equation*}
with $\mathcal{L}_0(\theta)=1$.
Hence, the Eq. (\ref{laplac}) is nothing but the backward equation for GBDP with scaled parameters $\lambda_{ik}^*$ and $\mu_{jk}^*$ and the distribution of $\mathcal{W}_k$ is known whenever the distribution of $Z_k$ is known.
\begin{remark}
	For $k_1=k_2=1$, (\ref{laplac}) reduces to 
	\begin{equation*}
		\theta\mathcal{W}_k(\theta)=\lambda_{1k}^*\mathcal{W}_{k+1}(\theta)+\mu_{1k}^*\mathcal{W}_{k-1}(\theta)-(\lambda_{1k}^*+\mu_{1k}^*)\mathcal{W}_k(\theta),
	\end{equation*}
with $\mathcal{W}_0(\theta)=1$, which coincide with result obtained in McNeil (1970).
\end{remark} 
\begin{remark}
In the cases of GLBDP, that is, $\lambda_{(k)_i}=k\lambda_i$ and $\mu_{(k)_j}=k\mu_j$, (\ref{laplac}) reduces to
\begin{equation*}
\theta\mathcal{W}_k(\theta)=\sum_{i=1}^{k_1}\tilde{\lambda}_{ik}^*\mathcal{W}_{k+i}(\theta)+\sum_{j=1}^{k_2}\tilde{\mu}_{jk}^*\mathcal{W}_{k-j}(\theta)-\bigg(\sum_{i=1}^{k_1}\tilde{\lambda}_{ik}^*+\sum_{j=1}^{k_2}\tilde{\mu}_{jk}^*\bigg)\mathcal{W}_k(\theta),
\end{equation*}
where
$
	\tilde{\lambda}_{ik}^*={k\lambda_i}/{g(k)}$ and $\tilde{\mu}_{jk}^*={k\mu_j}/{g(k)}.
$
\end{remark}
\subsection{Stochastic path integral of GBDP}
Let us consider the stochastic path integral $	\mathcal{X}(t)=\int_{0}^{t}g(\mathcal{N}(s))\,\mathrm{d}s$ of GBDP, where $g(\cdot)$ is a non-negative real valued function. Then, conditional on the event $\{\mathcal{N}(0)=m\}$, the joint distribution function $q_m(n,x,t)=\mathrm{Pr}\{\mathcal{N}(t)=n, \mathcal{X}(t)\leq x|\mathcal{N}(0)=m\}$, $t\ge0$ of $\mathcal{N}(t)$ and $\mathcal{X}(t)$ solves
\begin{align}
	\frac{\partial}{\partial t}q_m(n,x,t)+g(n)\frac{\partial}{\partial x}q_m(n,x,t)&=-\left(\sum_{i=1}^{k_1}\lambda_{(n)_i}+\sum_{j=1}^{k_2}\mu_{(n)_j}\right)q_m(n,x,t)+\sum_{i=1}^{k_1}\lambda_{(n-i)_i}q_m(n-i,x,t)\nonumber\\
	&\ \ +\sum_{j=1}^{k_2}\mu_{(n+j)_j}q_m(n+j,x,t),\ n\ge0,\label{intgbdp}
\end{align}
where $q_m(n-i,x,t)=0$ for all $i>n$. This follows from the transitions probabilities of GBDP and its Markov property.

\begin{remark}
		Let $\mathcal{T}_k(t)$ be the total time in $(0,t)$ during which the population size is $k$. Then, for $g(x)=x$, the stochastic path integral of GBDP can be written as
	\begin{equation*}
		X(t)=\int_{0}^{t}\mathcal{N}(s)\,\mathrm{d}s=\sum_{k=1}^{\infty}k\mathcal{T}_k(t).
	\end{equation*}
\end{remark}

Next, we consider the stochastic path integral of GLBDP.
\subsubsection{A linear case of GBDP} Here, we obtain some results related to the stochastic path integral of GLBDP. 
Let $X(t)$ be the stochastic path integral of GLBDP defined as \begin{equation*}
	X(t)=\int_{0}^{t}N(s)\,\mathrm{d}t,\ t\ge0.
\end{equation*}
From (\ref{intgbdp}), the system of differential equations that governs the joint distribution $	p_m(n,x,t)=\mathrm{Pr}\{N(t)=n,X(t)<x|N(0)=m\}$, $t\ge0$ of $N(t)$ and $X(t)$ is given by

\begin{align}
	\frac{\partial}{\partial t}p_m(n,x,t)+n\frac{\partial}{\partial x}p_m(n,x,t)&=-n\left(\sum_{i=1}^{k_1}\lambda_i+\sum_{j=1}^{k_2}\mu_j\right)p_m(n,x,t)+\sum_{i=1}^{k_1}(n-i)\lambda_ip_m(n-i,x,t)\nonumber\\
	&\ \ +\sum_{j=1}^{k_2}(n+j)\mu_jp_m(n+j,x,t),\ n\ge0.\label{intls}
\end{align}
Thus, their joint pgf  $\tilde{G}(u,v,t)=\int_{0}^{\infty}\sum_{n=0}^{\infty}u^nv^xp_m(n,x,t)\,\mathrm{d}x$, $|u|\leq1$, $|v|\leq1$ solves
\begin{align*}
\frac{\partial}{\partial t}\tilde{G}(u,v,t)=\left(\sum_{i=1}^{k_1}\lambda_i(u^i-1)+\sum_{j=1}^{k_2}\mu_j(u^{-j}-1)+\ln v\right)u\frac{\partial}{\partial u}\tilde{G}(u,v,t)
\end{align*}
with initial condition $\tilde{G}(u,v,0)=u^m$. Hence, the  cgf
$ \tilde{K}(\theta_u,\theta_v,t)= \ln\tilde{G}(e^{\theta_u},e^{\theta_v},t)$ solves
\begin{equation}\label{Z}
		\frac{\partial}{\partial t}\tilde{K}(\theta_u,\theta_v,t)=\left(\sum_{i=1}^{k_1}\lambda_i(e^{i\theta_u}-1)+\sum_{j=1}^{k_2}\mu_j(e^{-j\theta_u}-1)+\theta_v\right)\frac{\partial}{\partial \theta_u}\tilde{K}(\theta_u,\theta_v,t)
\end{equation}
with $\tilde{K}(\theta_u,\theta_v,0)=m\theta_u.$

Next, we obtain the mean and variance of the stochastic path integral of GLBDP and its covariance with $N(t)$. 
\begin{proposition}\label{prop6}
For $m=1$ and $t\ge0$, we have\\
\noindent (i) $
		\mathbb{E}(X(t))=\begin{cases}
			\frac{(e^{\eta t}-1)}{\eta},\ \eta\ne0,\vspace{0.1cm}\\
			t,\ \eta=0,
		\end{cases}
	$\\
	\vspace{0.1cm}
	\noindent (ii) $\mathrm{Cov}(N(t),X(t))=\begin{cases}
		\frac{\zeta}{\eta^2}(e^{2\eta t}-e^{\eta t}-\eta te^{\eta t}),\ \eta\ne0,\vspace{0.1cm}\\
	\zeta\frac{t^2}{2},\ \eta=0,
	\end{cases},$\\
	\vspace{0.1cm}
	\noindent (iii) $	\mathrm{Var}(X(t))=\begin{cases}
		\frac{2\zeta}{\eta^2}\left(\frac{e^{2\eta t}-1}{2\eta}-te^{\eta t}\right),\ \eta\ne0,\vspace{0.1cm}\\
	\zeta\frac{t^3}{3},\ \eta=0.
	\end{cases}$
\end{proposition}
\begin{proof}
	On using (\ref{cgfequ}) in (\ref{Z}), we have
	\begin{align}
		\frac{\partial}{\partial t}(\theta_u\mathbb{E}(N(t))&+\theta_v\mathbb{E}(X(t))+\frac{\theta_u^2}{2}\mathrm{Var}(N(t))+\frac{\theta_v^2}{2}\mathrm{Var}(X(t))+\theta_u\theta_v\mathrm{Cov}(N(t),X(t))+\dots)\nonumber\\
		&-\theta_v\sum_{l=0}^{\infty}\frac{(-\theta_u)^l}{l!}(\mathbb{E}(N(t))+\theta_u\mathrm{Var}(N(t))+\theta_v\mathrm{Cov}(N(t),X(t))+\dots)\nonumber\\
		&=\sum_{l=1}^{\infty}\left(\sum_{i=1}^{k_1}\lambda_i\frac{(i\theta_u)^l}{l!}+\sum_{j=1}^{k_2}\mu_j\frac{(-j\theta_u)^l}{l!}\right).\label{**}
	\end{align}
	On comparing the coefficients on both sides of (\ref{**}), we get
\begin{align*}
	\frac{\mathrm{d}}{\mathrm{d}t}\mathbb{E}(X(t))&=\mathbb{E}(N(t)),\\
	\frac{\mathrm{d}}{\mathrm{d}t}\mathrm{Cov}(N(t),X(t))&=\mathrm{Var}(N(t))+\eta\mathrm{Cov}(N(t),X(t))
\end{align*}
and 
\begin{equation*}
	\frac{\mathrm{d}}{\mathrm{d}t}\mathrm{Var}(X(t))=2\mathrm{Cov}(N(t),X(t)).
\end{equation*}
On solving the above differential equations with initial conditions $\mathbb{E}(X(0))=0$, $\mathrm{Cov}(N(0),X(0))=0$ and $\mathrm{Var}(X(0))=0$, respectively, we get the required results.
\end{proof}
\begin{remark}
	The limiting behavior of the mean of $X(t)$ is given as follows:
	\begin{equation*}
		\lim_{t\to\infty}\mathbb{E}(X(t))=\begin{cases}
			-\frac{1}{\eta},\ \eta<0,\vspace{0.1cm}\\
			\infty,\ \eta\ge0.
		\end{cases}
	\end{equation*}
\end{remark}
The correlation coefficient $R_X(1)=\mathrm{Corr}(N(1),X(1))$ for different values of parameter are illustrated in Figure~\ref{fig3}.
\begin{figure}[htp]
	\includegraphics[width=13cm]{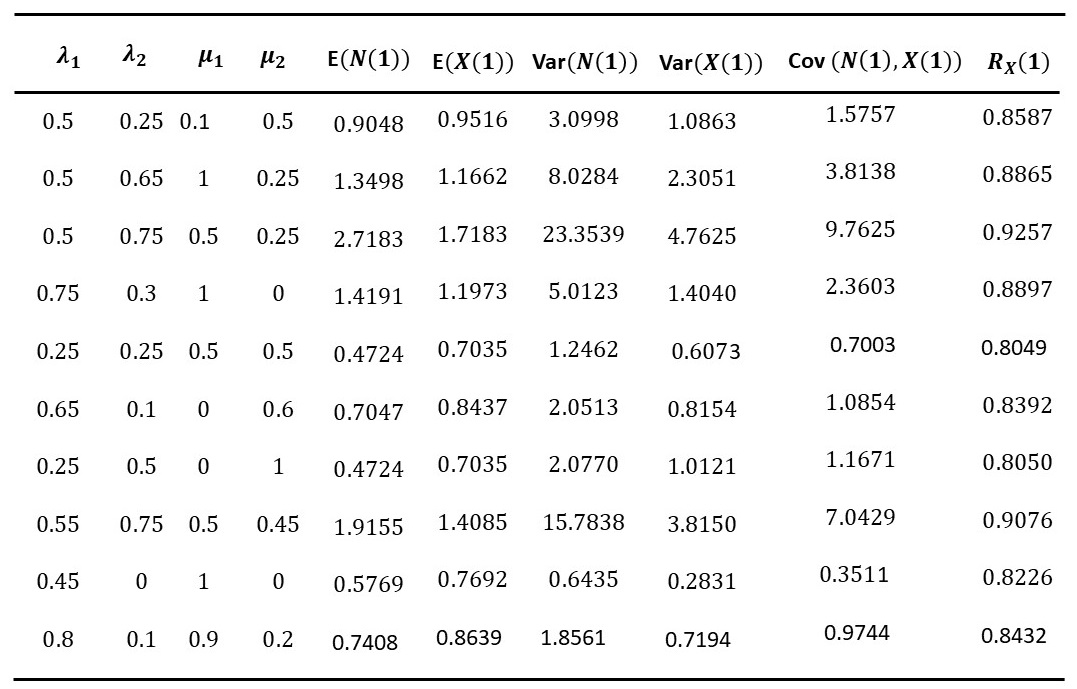}
	\caption[Figure 3]{\small Values of stochastic path integral of GLBDP at $t=1$.}\label{fig3}
\end{figure}

 The variation of correlation coefficient between $N(t)$ and its stochastic path integral with time is shown in 
 Figure~\ref{fig6}. For $\eta>0$, $N(t)$ and $X(t)$ have strong positive correlation and as $t$ increases it converges to $+1$. For $\eta<0$, the correlation is positive but it is getting weaker with time.
\begin{figure}[htp]
	\includegraphics[width=16cm]{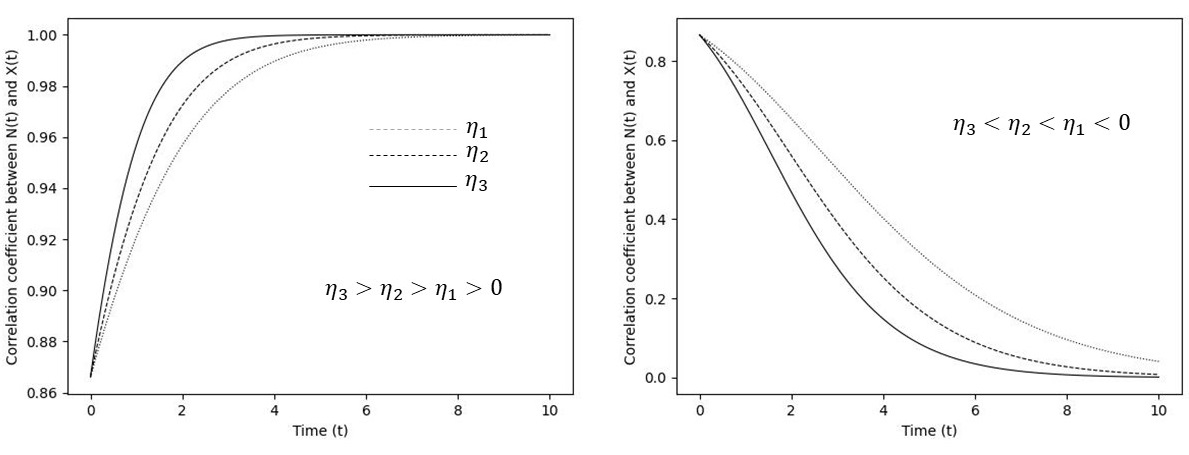}
	\caption[Figure 4]{Correlation plot of $N(t)$ and $X(t)$ for $k_1=k_2=2$.}\label{fig6}
\end{figure}
\subsection{Stochastic path integral of GBDP in the case of constant birth and death rates} Let $X^*(t)=\int_{0}^{t}N^*(s)\,\mathrm{d}s$ be the stochastic path integral  of the process $\{N^*(t)\}_{t\ge0}$ defined in Section~\ref{subsec1}. From (\ref{intgbdp}), the joint distribution $p^*_1(n,x,t)=\mathrm{Pr}\{N^*(t)=n,X^*(t)\leq x|N^*(0)=1\}$ solves
\begin{align}
	\frac{\partial}{\partial t}p_1^*(n,x,t)+n\frac{\partial}{\partial x}p_1^*(n,x,t)&=-\left(\sum_{i=1}^{k_1}\lambda_i+\sum_{j=1}^{k_2}\mu_j\right)p_1^*(n,x,t)+\sum_{i=1}^{k_1}\lambda_ip_1^*(n-i,x,t)\nonumber\\
	&\ \ +\sum_{j=1}^{k_2}\mu_jp_1^*(n+j,x,t),\ n\ge0\label{ge1}
\end{align}
with 
$p_1^*(1,0,0)=1$,
where $p^*(n,x,t)=0$ for all $n<0$.

The pgf $\tilde{G}^*(u,v,t)=\int_{0}^{\infty}\sum_{n=0}^{\infty}u^nv^x{p}_1^*(n,x,t)\,\mathrm{d}x$, $0<u<1,\ 0<v<1$ of $(N^*(t),X^*(t))$ solves the following differential equation:
\begin{equation}\label{intcon}
	\frac{\partial}{\partial t}\tilde{G}^*(u,v,t)-\ln v\frac{\partial}{\partial u}\tilde{G}^*(u,v,t)=\left(\sum_{i=1}^{k_1}\lambda_i(u^i-1)+\sum_{j=1}^{k_2}\mu_j(u^{-j}-1)\right)\tilde{G}^*(u,v,t)
\end{equation}
with initial condition $\tilde{G}^*(u,v,0)=u$. And, its cgf $\tilde{K}^*(\theta_u,\theta_v,t)=\ln\tilde{G}^*(e^{\theta_u},e^{\theta_v},t) $ solves 
\begin{equation}
	\frac{\partial}{\partial t}\tilde{K}^*(\theta_u,\theta_v,t)-\theta_ve^{-\theta_u}\frac{\partial}{\partial\theta_u}\tilde{K}^*(\theta_u,\theta_v,t)=\sum_{i=1}^{k_1}\lambda_i(e^{i\theta_u}-1)+\sum_{j=1}^{k_2}\mu_j(e^{-j\theta_u}-1)
\end{equation}
with $\tilde{K}^*(\theta_u,\theta_v,0)=\theta_u.$ On following the similar lines as in the proof of Proposition~\ref{prop6}, we get 
\begin{align*}
	\mathbb{E}(X^*(t))&=t+\frac{\eta t^2}{2},\\
	\mathrm{Cov}(N^*(t),X^*(t))&=t+\left(\eta-\sum_{i=1}^{k_1}i^2\lambda_i+\sum_{j=1}^{k_2}j^2\mu_j\right)\frac{t^2}{2}
\end{align*} 
and 
\begin{equation*}
	\mathrm{Var}(X^*(t))=t^2+\left(\eta-\sum_{i=1}^{k_1}i^2\lambda_i+\sum_{j=1}^{k_2}j^2\mu_j\right)\frac{t^3}{3}.
\end{equation*}

Next, we solve differential equation  (\ref{intcon}) to obtain the joint pgf of $N^*(t)$ and $X^*(t)$.

The subsidiary equation corresponding to  (\ref{intcon}) is as follows:  
	\begin{equation}\label{sub3}
		\frac{\partial t}{1}=\frac{\partial u}{-\ln v}=\frac{1}{\sum_{i=1}^{k_1}\lambda_i(u^i-1)+\sum_{j=1}^{k_2}\mu_j(u^{-j}-1)}\frac{\partial \tilde{G}^*}{\tilde{G}^*(u,v,t)}.
	\end{equation}
So, we get
	\begin{equation*}
		t\ln v  +u=\text{constant}
	\end{equation*}
	and 
	\begin{equation*}
		u\left(\sum_{i=1}^{k_1}\lambda_i\left(\frac{u^i}{i+1}-1\right)+\sum_{j=1}^{k_2}\mu_j\left(\frac{u^{-j}}{1-j}-1\right)\right)+\ln v\ln \tilde{G}^*(u,v,t)=\text{constant}.
	\end{equation*}
	Thus, the general solution of (\ref{intcon}) is given by
	\begin{equation*}
		u\left(\sum_{i=1}^{k_1}\lambda_i\left(\frac{u^i}{i+1}-1\right)+\sum_{j=1}^{k_2}\mu_j\left(\frac{u^{-j}}{1-j}-1\right)\right)+\ln v\ln \tilde{G}^*(u,v,t)=\Phi(t\ln v  +u),
	\end{equation*}
	where $\Phi$ is a real valued function. By
	using the initial condition $\tilde{G}^*(u,v,0)=u$, we obtain
	\begin{equation*}
		\Phi(u)=u\left(\sum_{i=1}^{k_1}\lambda_i\left(\frac{u^i}{i+1}-1\right)+\sum_{j=1}^{k_2}\mu_j\left(\frac{u^{-j}}{1-j}-1\right)\right)+\ln v\ln u.
	\end{equation*}
	Thus,
	\begin{equation*}
		\tilde{G}^*(u,v,t)=\exp\left(\frac{\Phi(t\ln v  +u)}{\ln v}-\frac{u}{\ln v}\left(\sum_{i=1}^{k_1}\lambda_i\left(\frac{u^i}{i+1}-1\right)+\sum_{j=1}^{k_2}\mu_j\left(\frac{u^{-j}}{1-j}-1\right)\right)\right).
	\end{equation*}

\section{Immigration effect in GLBDP}\label{sec6}	
In GLBDP, the state $n=0$ is the absorbing state of the process, that is, once the population get extinct it can not be revive again. Here, we consider a possibility that the process can revive again once it reaches state $n=0$. To capture the immigration effect in GLBDP, we introduce a linear version of GBDP, namely, the generalized linear birth-death process with immigration (GLBDPwI).

 Let $\nu>0$ be the rate of population growth due to immigration. Here, we consider two different cases of immigration as follows:
\subsection{Immigration effect at state $n=0$ only}
In this case, the immigration happens only when no alive individual is present in the population. So, the immigration rate is  $\lambda_{(0)_i}=\nu$, and once the immigration happens, the birth and death  rates are $\lambda_{(n)_{i}}=n\lambda_i$, $i\in\{1,2,\dots,k_1\}$ and $\mu_{(n)_j}$, $j\in\{1,2,\dots,k_2\}$, respectively, for all $n\ge1$. From (\ref{tranprob}), the state probabilities $p(n,t)$, $n\ge0$ of GLBDPwI solve the following system of differential equations:
\begin{align*}
	\frac{\mathrm{d}}{\mathrm{d}t}p(0,t)&=-k_1\nu p(0,t)+\sum_{j=1}^{k_2}j\mu_jp(j,t),\\
	\frac{\mathrm{d}}{\mathrm{d}t}p(n,t)&=-\left(\sum_{i=1}^{k_1}n\lambda_i+\sum_{j=1}^{k_2}n\mu_j\right)p(n,t)+\nu p(0,t)+\sum_{i=1}^{k_1}(n-i)\lambda_ip(n-i,t)\\
    &\ \ +\sum_{j=1}^{k_2}(n+j)\mu_jp(n+j,t),\ 1\leq n\leq k_1
\end{align*}
and 
\begin{equation*}
	\frac{\mathrm{d}}{\mathrm{d}t}p(n,t)=-\left(\sum_{i=1}^{k_1}n\lambda_i+\sum_{j=1}^{k_2}n\mu_j\right)p(n,t)+\sum_{i=1}^{k_1}(n-i)\lambda_ip(n-i,t)+\sum_{j=1}^{k_2}(n+j)\mu_jp(n+j,t),\ n>k_1,
\end{equation*}
with $p(1,0)=1$ and $p(n,t)=0$ for all $n\ne1$. Here, $p(n,t)=0$ for all $n<0$. Its pgf $\mathcal{H}(u,t)=\sum_{n=0}^{\infty}u^np(n,t)$ solves
\begin{equation}\label{0imi}
	\frac{\partial}{\partial t}\mathcal{H}(u,t)=\left(\sum_{i=1}^{k_1}\lambda_iu(u^i-1)+\sum_{j=1}^{k_2}\mu_ju(u^{-j}-1)\right)\frac{\partial}{\partial u}\mathcal{H}(u,t)+\sum_{i=1}^{k_1}\nu(u^i-1)p(0,t)
\end{equation}
with $\mathcal{H}(u,0)=u$.

On taking the derivative of (\ref{0imi}) with respect to $u$, we get
\begin{align}
	\frac{\partial^2}{\partial u\partial t}\mathcal{H}(u,t)&=\left(\sum_{i=1}^{k_1}\lambda_iu(u^i-1)+\sum_{j=1}^{k_2}\mu_ju(u^{-j}-1)\right)\frac{\partial^2}{\partial u^2}\mathcal{H}(u,t)\nonumber\\
	&\ \ +\left(\sum_{i=1}^{k_1}\lambda_iu((i+1)u^i-1)+\sum_{j=1}^{k_2}\mu_j((1-j)u^{-j}-1)\right)\frac{\partial}{\partial u}\mathcal{H}(u,t)+\sum_{i=1}^{k_1}i\nu u^{i-1}p(0,t).\label{7.2}
\end{align}
On taking $u=1$ in (\ref{7.2}), we get the following differential equation that governs the mean $\mathscr{M}(t)=\sum_{n=0}^{\infty}np(n,t)$ of GLBDPwI:
\begin{equation*}
	\frac{\mathrm{d}}{\mathrm{d}t}\mathscr{M}(t)=\eta \mathscr{M}(t)+\frac{k_1(k_1+1)}{2}\nu p(0,t)
\end{equation*}
with $\mathscr{M}(0)=1$, where $\eta=\sum_{i=1}^{k_1}i\lambda_i-\sum_{j=1}^{k_2}j\mu_j.$ Thus,
\begin{equation*}
	\mathscr{M}(t)=e^{\eta t}\left(1+\nu\frac{k_1(k_1+1)}{2}\int_{0}^{t}e^{-\eta s}p(0,s)\,\mathrm{d}s\right).
\end{equation*}
\subsection{Immigration effect at any state}
In this case, the immigration effect is always present in the population at a constant rate. So, the birth and death rates are $\lambda_{(n)_i}=n\lambda_i+\nu$, $i\in\{1,2,\dots,k_1\}$ and $\mu_{(n)_j}$, $j\in\{1,2,\dots,k_2\}$, respectively, for all $n\ge0$. From (\ref{tranprob}), the state probabilities $p(n,t)$ of GLBDPwI solve
\begin{align*}
	\frac{\mathrm{d}}{\mathrm{d}t}p(n,t)&=-\left(k_1\nu+n\sum_{i=1}^{k_1}\lambda_i+n\sum_{j=1}^{k_2}\mu_j\right)p(n,t)+\sum_{i=1}^{k_1}(\nu+(n-i)\lambda_i)p(n-i,t)\\
	&\ \ +\sum_{j=1}^{k_2}(n+j)\mu_jp(n+j,t),\ n\ge0,
\end{align*}
where $p(1,0)=1$ and $p(n,t)=0$ for all $n\ne1$. So, the differential equation that governs the pgf $\mathcal{H}(u,t)=\sum_{n=0}^{\infty}u^np(n,t)$ of GLBDPwI is given by
\begin{equation}\label{imipgf}
	\frac{\partial}{\partial t}\mathcal{H}(u,t)=\left(\sum_{i=1}^{k_1}\lambda_iu(u^i-1)+\sum_{j=1}^{k_2}\mu_ju(u^{-j}-1)\right)\frac{\partial}{\partial u}\mathcal{H}(u,t)+\left(\sum_{i=1}^{k_1}\nu(u^i-1)\right)\mathcal{H}(u,t)
\end{equation}
with $\mathcal{H}(u,0)=u$.
 
On taking the derivative with respect to $u$ on both sides of (\ref{imipgf}), we have
\begin{align}
	\frac{\partial}{\partial t}\frac{\partial}{\partial u}\mathcal{H}(u,t)&=\left(\sum_{i=1}^{k_1}\lambda_iu(u^i-1)+\sum_{j=1}^{k_2}\mu_ju(u^{-j}-1)\right)\frac{\partial^2}{\partial u^2}\mathcal{H}(u,t)\nonumber\\
	&\ \ +\left(\sum_{i=1}^{k_1}\lambda_i((i+1)u^i-1)+\sum_{j=1}^{k_2}\mu_j((-j+1)u^{-j}-1)\right)\frac{\partial}{\partial u}\mathcal{H}(u,t)+\sum_{i=1}^{k_1}i\nu u^{i-1}\mathcal{H}(u,t)\nonumber\\
	&\ \ +\sum_{i=1}^{k_1}\nu(u^i-1)\frac{\partial}{\partial u}\mathcal{H}(u,t).\label{imimean}
\end{align}
On taking $u=1$ in (\ref{imimean}), we get the mean $\mathscr{M}(t)$ of GLBDPwI as a solution of the following differential equation:
\begin{equation*}
	\frac{\partial}{\partial t}\mathscr{M}(t)=\eta \mathscr{M}(t)+\frac{k_1(k_1+1)\nu}{2}
\end{equation*}
with initial condition $\mathscr{M}(0)=1$. It is given by
\begin{equation*}
	\mathscr{M}(t)=\begin{cases}
		e^{\eta t}+\frac{k_1(k_1+1)\nu}{2\eta}(e^{\eta t}-1),\ \eta\ne0,\vspace{0.1cm}\\
		1+\frac{k_1(k_1+1)\nu t}{2},\ \eta=0,
	\end{cases}
\end{equation*}
where $\eta=\sum_{i=1}^{k_1}i\lambda_i-\sum_{j=1}^{k_2}j\mu_j$.
\begin{remark}
	For $\nu=0$, the mean of GLBDPwI  reduces to the mean of GLBDP given in (\ref{meanvar}).
\end{remark}

For different values of $\eta$ the variation of expected values of GLBDP and GLBDPwI with time is shown in Figure~\ref{fig7}.
\begin{figure}[htp]
	\includegraphics[width=16cm]{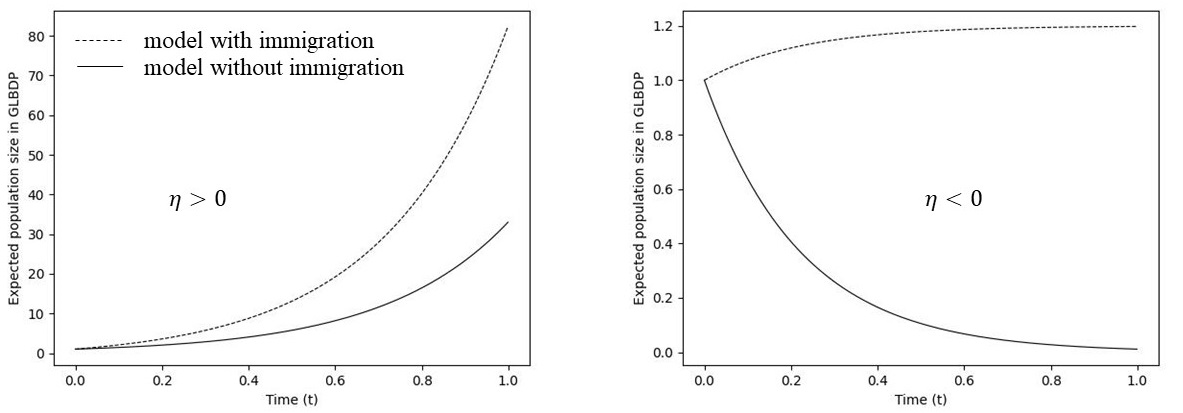}
	\caption[Figure 7]{Expected population size in GLBDP and GLBDPwI for $k_1=k_2=2$.}\label{fig7}
\end{figure}
In GLBDPwI, for $\eta>0$, the population size grow exponentially fast for the large value of $t$ and for $\eta<0$, the limiting mean value for large $t$ is $-k_1(k_1+1)\nu/2\eta$, that is,
\begin{equation*}
	\lim_{t\to\infty}\mathscr{M}(t)=\begin{cases}
		\infty,\ \eta\ge0,\vspace{0.1cm}\\
		-\frac{k_1(k_1+1)\nu}{2\eta},\ \eta<0.
	\end{cases}
\end{equation*}
Thus, as expected in GLBDPwI, the population will never go extinct.
 \begin{remark}
	On comparing (\ref{pgfequ}) and (\ref{imipgf}), we note that they differs in the term $\sum_{i=1}^{k_1}\nu(u^i-1)\mathcal{H}(u,t)$. It is a  generalized counting process component with constant transition rates. Suppose that the birth rates in GLBDPwI tend to zero, that is, $\lambda_i\to0$ for all $i\in\{1,2,\dots,k_1\}$ then it can be used to study the fluctuation in the number of particles enclosed inside very small volume. Here, the immigration of particle from surrounding happens as GCP with constant rates and the emigration of finitely many particles from closed volume is considered as death whose rates are  linearly proportional to the number of particle present. 
\end{remark}
\section{Application to vehicles parking management system}
Suppose a vehicles parking lot has maximum capacity of $K$ parking spots. If we assume that the gate of the parking area is wide enough such that multiple vehicles can arrive or depart together and the chance of simultaneous arrival and departure of the vehicles is negligible.  Let $N(t)$ denotes the total number of vehicles parked at time $t$. Then, we can use a linear version of GBDP to model this problem, where $\{N(t)\}_{t\ge0}$ is a GBDP with parameters
\begin{equation*}
	\lambda_{(n)_i}=(K-n)\lambda_i,\ i\in\{1,2,\ldots,K_1\}\ \text{and}\  \mu_{(n)_j}=n\mu_j,\ j\in \{1,2,\ldots,K_2\}
\end{equation*}
for all $0\leq n\leq K$, $K_1< K$ and $K_2<K$.
Thus, the rates of arrival and departure of vehicles are proportional to the number of vacant spots and the total number of vehicles parked in the parking lot, respectively. Let $A(t)$ and $D(t)$ denote the number of vehicles arrived  and departed by time $t$. 
Here, we are interested in knowing the quantities $\mathbb{E}(N(t))$, $\mathbb{E}(A(t))$ and $\mathbb{E}(D(t))$. Let us assume that there are no vehicles parked at time $t=0$. As obtained in the case of GLBDP, it can be shown that the pmf $p(n,t)=\mathrm{Pr}\{N(t)=n\}$ solves the following system of differential equations:
\begin{align*}
	\frac{\mathrm{d}}{\mathrm{d}t}p(n,t)&=-\bigg(\sum_{i=1}^{K_1}(K-n)\lambda_i+\sum_{j=1}^{K_2}n\mu_j\bigg)p(n,t)+\sum_{i=1}^{K_1}(K-n+i)\lambda_ip(n-i,t)\\
	&\ \ +\sum_{j=1}^{K_2}(n+j)\mu_jp(n+j,t),\ 0\leq n\leq K,
\end{align*}
with initial condition $p(0,0)=1$.
Here, $p(n,t)=0$ for all $n<0$ and $n>K$.
Hence, the governing system of differential equation for its mean is given by
\begin{equation}\label{apl1}
	\frac{\mathrm{d}}{\mathrm{d}t}\mathbb{E}(N(t))=-\left(\sum_{i=1}^{K_1}i\lambda_i+\sum_{j=1}^{K_2}j\mu_j\right)\mathbb{E}(N(t))+K\sum_{i=1}^{K_1}i\lambda_i
\end{equation} 
with initial condition $\mathbb{E}(N(0))=0$. On solving (\ref{apl1}), we get
\begin{equation*}
	\mathbb{E}(N(t))=\frac{K\sum_{i=1}^{K_1}i\lambda_i}{\beta}\left(1-e^{-\beta t}\right),
\end{equation*}
where $\beta=\sum_{i=1}^{K_1}i\lambda_i+\sum_{j=1}^{K_2}j\mu_j$. In the long run, the expected number of vehicles parked is given by
\begin{equation*}
	\lim_{t\to\infty}\mathbb{E}(N(t))=\frac{K\sum_{i=1}^{K_1}i\lambda_i}{\beta}<K.
\end{equation*}

The joint pmf $p(a,n,t)=\mathrm{Pr}\{A(t)=a,N(t)=n\}$ solves
\begin{align*}
	\frac{\mathrm{d}}{\mathrm{d}t}p(a,n,t)&=-\left(\sum_{i=1}^{K_1}(K-n)\lambda_i+\sum_{j=1}^{K_2}n\mu_j\right)p(a,n,t)+\sum_{i=1}^{K_1}(K-n+i)\lambda_ip(a-i,n-i,t)\\
	&\ \ +\sum_{j=1}^{K_2}(n+j)\mu_jp(a,n+j,t),\ a,n\in\{0,1,\dots,K\}
\end{align*}
with $p(0,0,0)=1$. Its proof follows similar lines to that of Proposition~\ref{prop3.1}.

The joint cgf $K(\theta_u,\theta_v,t)=\ln \mathbb{E}(e^{\theta_uA(t)+\theta_vN(t)})$ satisfies the following differential equation:
\begin{equation*}
	\frac{\partial}{\partial t}K(\theta_u,\theta_v,t)+\left(\sum_{i=1}^{K_1}\lambda_i(e^{i(\theta_u+\theta_v)}-1)-\sum_{j=1}^{K_2}\mu_j(e^{-j\theta_v}-1)\right)\frac{\partial}{\partial \theta_v}K(\theta_u,\theta_v,t)=K\sum_{i=1}^{K_1}\lambda_i(e^{i(\theta_u+\theta_v)}-1)
\end{equation*}
with initial condition $K(\theta_u,\theta_v,0)=0$. On using (\ref{cgfequ}), we get
\begin{equation*}
	\frac{\mathrm{d}}{\mathrm{d}t}\mathbb{E}(A(t))=\sum_{i=1}^{K_1}i\lambda_i(K-\mathbb{E}(N(t)))
\end{equation*}
with $\mathbb{E}(A(0))=0$. Thus,
\begin{equation*}
	\mathbb{E}(A(t))=\sum_{i=1}^{K_1}i\lambda_i\left(Kt-\frac{K\sum_{i=1}^{K_1}i\lambda_i}{\beta}\bigg(t+\frac{e^{-t\beta}-1}{\beta}\bigg)\right).
\end{equation*}

Similarly, the joint pmf $q(d,n,t)=\mathrm{Pr}\{D(t)=d,N(t)=n\}$ is the solution of following system of differential equations:
\begin{align}
	\frac{\mathrm{d}}{\mathrm{d}t}q(d,n,t)&=-\left(\sum_{i=1}^{K_1}(K-n)\lambda_i+\sum_{j=1}^{K_2}n\mu_j\right)q(d,n,t)+\sum_{i=1}^{K_1}(K-n+i)\lambda_iq(d,n-i,t)\nonumber\\
	&\ \ +\sum_{j=1}^{K_2}(n+j)\mu_jq(d-j,n+j,t),\ d,n\in\{0,1,\dots,K\}.\nonumber
\end{align}
Hence, the cgf $\mathcal{K}(\theta_u,\theta_v,t)=\ln \mathbb{E}(e^{\theta_uD(t)+\theta_vN(t)})$ solves
\begin{equation*}
	\frac{\partial}{\partial t}	\mathcal{K}(\theta_u,\theta_v,t)+\left(\sum_{i=1}^{K_1}\lambda_i(e^{i\theta_v}-1)+\sum_{j=1}^{K_2}\mu_j(e^{j(\theta_u-\theta_v)}-1)\right)\frac{\partial}{\partial \theta_v}\mathcal{K}(\theta_u,\theta_v,t)=K\sum_{i=1}^{K_1}\lambda_i(e^{i\theta_v}-1)
\end{equation*}
with $\mathcal{K}(\theta_u,\theta_v,0)=0$. Thus, we get
\begin{equation*}
	\mathbb{E}(D(t))=K\sum_{i=1}^{K_1}i\lambda_i\sum_{j=1}^{K_2}j\mu_j\left(\frac{t}{\beta}+\frac{e^{-\beta t}-1}{\beta^2}\right).
\end{equation*}

Let us consider the stochastic path integral $X(t)=\int_{0}^{t}N(s)\,\mathrm{d}s$. Then, we can find the average occupancy $O(t)$ of the parking lot at any given time $t\ge0$ which we defined as $	O(t)=t^{-1}\mathbb{E}(X(t))$.

From (\ref{intgbdp}), the system of differential equations that governs the joint distribution $p_0(n,x,t)$ is given by
\begin{align*}
	\frac{\partial}{\partial t}p_0(n,x,t)+n\frac{\partial}{\partial x}p_0(n,x,t)&=-\left(\sum_{i=1}^{K_1}(K-n)\lambda_i+\sum_{j=1}^{K_2}n\mu_j\right)p_0(n,x,t)\\
	&\ \ +\sum_{i=1}^{K_1}(K-n+i)\lambda_ip_0(n-i,x,t)\\
	&\ \ +\sum_{j=1}^{K_2}(n+j)\mu_jp_0(n+j,x,t),\ 0\leq n\leq K,
\end{align*}
where $p_0(n,x,t)=0$ for all $n<0$ and $n>K$.

So, the joint cgf $\tilde{K}(\theta_u,\theta_v,t)=\ln \mathbb{E}(e^{\theta_uN(t)+\theta_vX(t)})$ solves
\begin{equation}\label{***}
	\frac{\partial}{\partial t}\tilde{K}(\theta_u,\theta_v,t)+\left(\sum_{i=1}^{K_1}\lambda_i(e^{i\theta_u}-1)-\sum_{j=1}^{K_2}\mu_j(e^{-j\theta_u}-1)-\theta_v\right)\frac{\partial}{\partial\theta_u}\tilde{K}(\theta_u,\theta_v,t)=K\sum_{i=1}^{K_1}\lambda_i(e^{i\theta_u}-1)
\end{equation}
with $\tilde{K}(\theta_u,\theta_v,0)=0$.

On using (\ref{cgfequ}), we get the following differential equation:
\begin{equation*}
	\frac{\mathrm{d}}{\mathrm{d}t}\mathbb{E}(X(t))=\mathbb{E}(N(t))
\end{equation*}
with $\mathbb{E}(X(0))=0$. Hence, 
\begin{equation*}
	\mathbb{E}(X(t))=\frac{K\sum_{i=1}^{K_1}i\lambda_i}{\beta}\left(t+\frac{e^{-\beta t}-1}{\beta}\right).
\end{equation*}
Thus, average occupancy is
\begin{equation*}
	O(t)=\frac{K\sum_{i=1}^{K_1}i\lambda_i}{\beta}\left(1+\frac{e^{-\beta t}-1}{\beta t}\right).
\end{equation*}

\end{document}